\documentclass{amsart}
\usepackage{fullpage}

\usepackage{color}
\usepackage{amssymb, amsmath, amsthm}
\usepackage{amsrefs}

\usepackage{graphicx}
\usepackage{pinlabel}
\usepackage[colorlinks=true, allcolors=blue]{hyperref}
\usepackage[foot]{amsaddr}

\usepackage{soul} %for highlighting a number in a calculation

\newtheorem{theorem}{Theorem}[section]
\newtheorem*{theorem*}{Theorem}
\newtheorem{corollary}[theorem]{Corollary}
\newtheorem*{corollary*}{Corollary}

\newtheorem{lemma}[theorem]{Lemma}
\newtheorem{proposition}[theorem]{Proposition}

\newtheorem*{theorem-Structure}{Theorem \ref{GCT}}

\theoremstyle{definition}
\newtheorem{definition}[theorem]{Definition}

\newtheorem{remark}[theorem]{Remark}

\newcommand{\N}{\mathbb{N}}

\newcommand{\co}{\mskip0.5 mu\colon\thinspace} 
\newcommand{\nil}{\varnothing}

\newcommand{\wihat}[1]{\widehat{#1}}

\newcommand{\defn}[1]{\textbf{#1}}

\newcommand{\g}{\operatorname{g}}

\newcommand{\oc}[1]{\vect{\textbf{c}}}

\newcommand{\netchi}{\operatorname{net}\chi} %net euler characteristic
\newcommand{\boundary}{\partial}

\newcommand{\mc}[1]{\mathcal{#1}}

\newcommand{\spacing}{
\parskip 6.6pt
\parindent 0pt
}

\spacing

\begin{document}

\title{Strong Haken via Thin Position}
\author{Scott A. Taylor}
\email{sataylor@colby.edu}

\dedicatory{For Mario Eudave-Mu\~noz}

\begin{abstract}
We use thin position of Heegaard splittings to give a new proof of Haken's Lemma that a Heegaard surface of a reducible manifold is reducible and of Scharlemann's ``Strong Haken Theorem'': a Heegaard surface for a 3-manifold may be isotoped to intersect a given collection of essential spheres and discs in a single loop each. We also give a reformulation of Casson and Gordon's theorem on weakly reducible Heegaard splittings, showing that they exhibit additional structure with respect to certain incompressible surfaces. This article could also serve as an introduction to the theory of generalized Heegaard surfaces and it includes a careful study of their behaviour under amalgamation.
\end{abstract}

\maketitle

\section{Introduction}
A surface $H$ in a closed, connected, orientable 3-manifold $M$ is a \defn{Heegaard surface} if $M = H_1 \cup_H H_2$ where $H_1$ and $H_2$ are handlebodies and (after gluing) $H = \boundary H_1 = \boundary H_2$. This definition can be generalized to 3-manifolds with boundary, by allowing $H$ to divide $M$ into two compressionbodies, rather than handlebodies. The \defn{Heegaard genus} $\g(M)$ of $M$ is the minimal genus of a Heegaard surface in $M$. A 3-manifold $M$ is \defn{reducible} if there is an embedded sphere $S \subset M$ such that $S$ is not the boundary of a 3-ball in $M$; the sphere $S$ is a \defn{reducing sphere}. A Heegaard surface is \defn{reducible} if there exists a reducing sphere intersecting it in a single simple closed curve. 

One of the fundamental results of Heegaard splitting theory is:

\begin{theorem*}[{Haken's Lemma (1968) \cite{Haken}}]
Every Heegaard surface $H$ for a compact, connected, orientable reducible 3-manifold $M$ is reducible.
\end{theorem*}

From this it follows easily that:
\begin{corollary*}[Heegaard genus is additive]
If $M_1$ and $M_2$ are compact and orientable, then $\g(M_1 \# M_2) = \g(M_1) + \g(M_2)$.
\end{corollary*}

A sphere $S \subset M$ is \defn{essential} if it is a reducing sphere that is not $\boundary$-parallel in $M$. If $\boundary M$ does not contain a sphere, then every sphere in $M$ is a reducing sphere if and only if it is essential. If $M$ is not a 3-ball but $\boundary M$ contains a sphere, then any $\boundary$-parallel sphere is also a reducing sphere. To avoid trivialities, we will generally phrase our results for essential spheres, rather than reducing spheres. A properly embedded disc $D$ in $M$ is \defn{essential} if $\boundary D$ does not bound a disc in $\boundary M$. Casson and Gordon \cite{CG} adapted Haken's result so that it also applies to discs:

\begin{lemma}[Disc version of Haken's Lemma]\label{HL-disc}
If $H$ is a Heegaard surface for a compact orientable 3-manifold $M$ with boundary such that $M$ contains an essential disc $D$, then there exists an essential disc $D'$ with the same boundary as $D$, such that $D$ and $H$ intersect in a single simple closed curve.
\end{lemma}

As with Haken's Lemma, this has the corollary:
\begin{corollary*}
If $M$ is the boundary connected sum of 3-manifolds $M_1$, and $M_2$, then $\g(M) = \g(M_1) + \g(M_2)$. Likewise, if $M$ is obtained by attaching a 1--handle to the the boundary of a compact, connected, orientable $M'$, then $\g(M) = \g(M') + 1$.
\end{corollary*}

Noting that Haken's result only guarantees existence, in 2020, Scharlemann considered the question of whether if $S$ is a fixed essential sphere or disc, a Heegaard surface $H$ can necessarily be isotoped to intersect the given sphere or disc $S$ in a single simple closed curve. He showed:

\begin{theorem*}[{Strong Haken Theorem \cite{Scharlemann}}]
Suppose that $\mc{S}$ is the union of pairwise disjoint properly embedded essential spheres and discs in a compact, connected, orientable 3-manifold $M$. If $H$ is a Heegaard surface for $M$, then $H$ can be isotoped so that $H$ intersects each component of $\mc{S}$ in a single simple closed curve.
\end{theorem*}

Scharlemann's proof is quite involved and proceeds by considering the intersections between a spine for one of the handlebodies in the Heegaard splitting given by $H$ and the surfaces in the collection $\mc{S}$. In \cite{ST-Wald} (see also \cite{SSS}), Scharlemann and Thompson use a similar perspective to prove Haken's Lemma and Waldhausen's classification of Heegaard splittings of $S^3$ \cite{Waldhausen-3sph}. (Another earlier proof of Haken's Lemma using normal surface theory was given by Johannson \cite{Johannson}.) Scharlemann's approach also allows him to prove a version of the theorem when $\boundary M$ contains spheres. That adds some complication, since it is then possible for an essential sphere to be altogether disjoint from a Heegaard surface.

In 2021, Hensel and Schultens \cite{HeSc} gave a much shorter proof of the Strong Haken Theorem for a closed 3-manifold $M$. Their proof uses ``sphere complexes,'' a construction arising from recent work in geometric group theory. In this paper, we give another, completely different, proof of the Strong Haken Theorem using thin position for 3-manifolds; an idea that goes back to Scharlemann-Thompson \cite{ST}. Although our proof occupies more pages, there is a certain simplicity to the basic idea. Along the way we develop some Heegaard splitting techniques that may be of independent interest. It is also worth noting that results such as the Strong Haken Theorem are by no means obvious. Recently, the author \cite{Taylor} showed that there is no such result if $H$ and $\mc{S}$ are required to be invariant under a smooth finite group action. Indeed, equivariant Heegaard genus can be sub-additive, additive, or super-additive under equivariant connect sums of 3-manifolds. 

Our approach is rooted in the Casson-Gordon Theorem for Heegaard surfaces \cite{CG}, and, in turn, gives a new perspective on it. A Heegaard surface is \defn{weakly reducible} if it admits disjoint compressing discs on opposites sides. Such a pair of discs is a \defn{weak reducing pair}. A Heegaard surface is \defn{strongly irreducible} if it is not weakly reducible. A Heegaard surface $H$ is \defn{stabilized} if there are compressing discs for $H$ on opposite sides of $H$ and with boundaries transversally intersecting in a single point. Such a pair of discs is a \defn{stabilizing pair}.
\begin{theorem*}[Casson-Gordon]
If $H$ is a weakly reducible Heegaard surface for a closed 3-manifold $M$, then either $H$ is stabilized or $M$ contains a closed essential surface.
\end{theorem*}

It follows from their techniques that the closed essential surface has genus at most that of $H$. In their context, ``essential'' means that the surface is incompressible and not a 2-sphere bounding a 3-ball. They actually show that if $H$ is weakly reducible, then either $H$ is reducible or $M$ contains a closed essential surface. Waldhausen \cite{Waldhausen-3sph} showed that every positive genus Heegaard surface of $S^3$ is stabilized; thus the existence of an inessential sphere intersecting $H$ in a single simple closed curve implies that $H$ is reducible. 

The situation when $\boundary M \neq \nil$ is somewhat more complicated and there are counter-examples to the theorem if we naively drop the requirement that $M$ be closed. The possibility of a Casson-Gordon theorem for 3-manifolds with boundary was first studied by Moriah \cite{Moriah}, Moriah-Sedgwick \cite{MS}, and Kobayashi \cite{Kobayashi2}. Hayashi-Shimokawa \cite{HS} improved on that work\footnote{Hayashi-Shimokawa use the term \emph{netted} instead of \emph{$\boundary$-stabilized}.} and proved:
\begin{theorem*}[Hayashi-Shimokawa]
If $H$ is a weakly reducible Heegaard surface for a compact, orientable 3-manifold (possibly with boundary) $M$, then either $H$ is stabilized, $H$ is $\boundary$-stabilized, or $M$ contains a closed essential surface.
\end{theorem*}

One way of defining $\boundary$-stabilization is as follows. Let $M$ be a compact, orientable 3-manifold having a component $F \subset \boundary M$ of positive genus. Let $J$ be a Heegaard surface for $M$ and let $H'$ be a parallel copy of $F$ separating $F$ from $J$. Let $H$ be the result of tubing $J$ to $H'$ via a tube that is vertical in the compressionbody between $J$ and $H'$. We say that $H$ is obtained by a $\boundary$-stabilization of $J$ along $F$. Alternatively, one can show that $H$ is $\boundary$-stabilized along $F$ if there exists a compressing disc $D$ for $H$ such that $H|_D$ consists of two surfaces, one of which is a Heegaard surface for $M$ and the other is parallel to $F$. The disc $D$ is called a \defn{$\boundary$-stabilizing} disc for $H$.

Hayashi-Shimokawa's theorem relies on the classification of Heegaard splittings of $(\text{surface}) \times I$  by Scharlemann-Thompson \cite{ST-classification}. Our work, also relying on the Waldhausen and Scharlemann-Thompson theorems, gives a somewhat different proof of Hayashi-Shimokawa's theorem. Indeed, we prove the following which recasts the Casson-Gordon theorem and the Hayashi-Shimokawa theorem in the mold of Haken's Lemma. In the following, as in the rest of the paper, an \defn{essential surface} in a 3-manifold is a surface that is incompressible, not $\boundary$-parallel, and not a 2-sphere bounding a 3-ball; every 2-sphere is incompressible, but perhaps not essential. The statement we give here is detailed, but hopefully of interest. See Figure \ref{fig: HakenSum} below for a depiction of the claim about Haken Sums. 

\begin{theorem-Structure}[Structure Theorem for Weakly Reducible Heegaard Surfaces]
Suppose that $H$ is an oriented weakly reducible Heegaard surface for a compact, orientable 3-manifold (possibly with boundary) $M$. Then there exists an oriented incompressible surface $F\subset M$ such that:
\begin{enumerate}
\item For each component $F'$ of $F$, $F' \cap H$ consists of a single simple closed curve, which is inessential in $F'$. The curves of $F \cap H$ bound discs in $F$ all on the same side of $H$; we may perform the construction so that the discs lie on which ever side we wish.
\item $F \cap H$ separates $H$.
\item For each component $F'$ of $F$, $-\chi(F') < -\chi(H)$.
\item There is a weak reducing pair of discs separated by $F$. 
\item If some component of $F$ is a 2-sphere bounding a 3-ball, then there is a stabilizing pair of discs for $H$ disjoint from $F$
\item If some component of $F$ is $\boundary$-parallel, then there is a $\boundary$-stabilizing disc for $H$ disjoint from $F$.
\item There is a generalized Heegaard surface $\mc{H}$ for $M$ such that $\mc{H}^+$ is isotopic to $H+F$, the Haken sum of $H$ and $F$, and $\mc{H}^- = F$; furthermore, the thick surfaces $\mc{H}^+$ are strongly irreducible.
\end{enumerate}
\end{theorem-Structure}

\begin{figure}[ht!]
\centering
\includegraphics[scale=0.12]{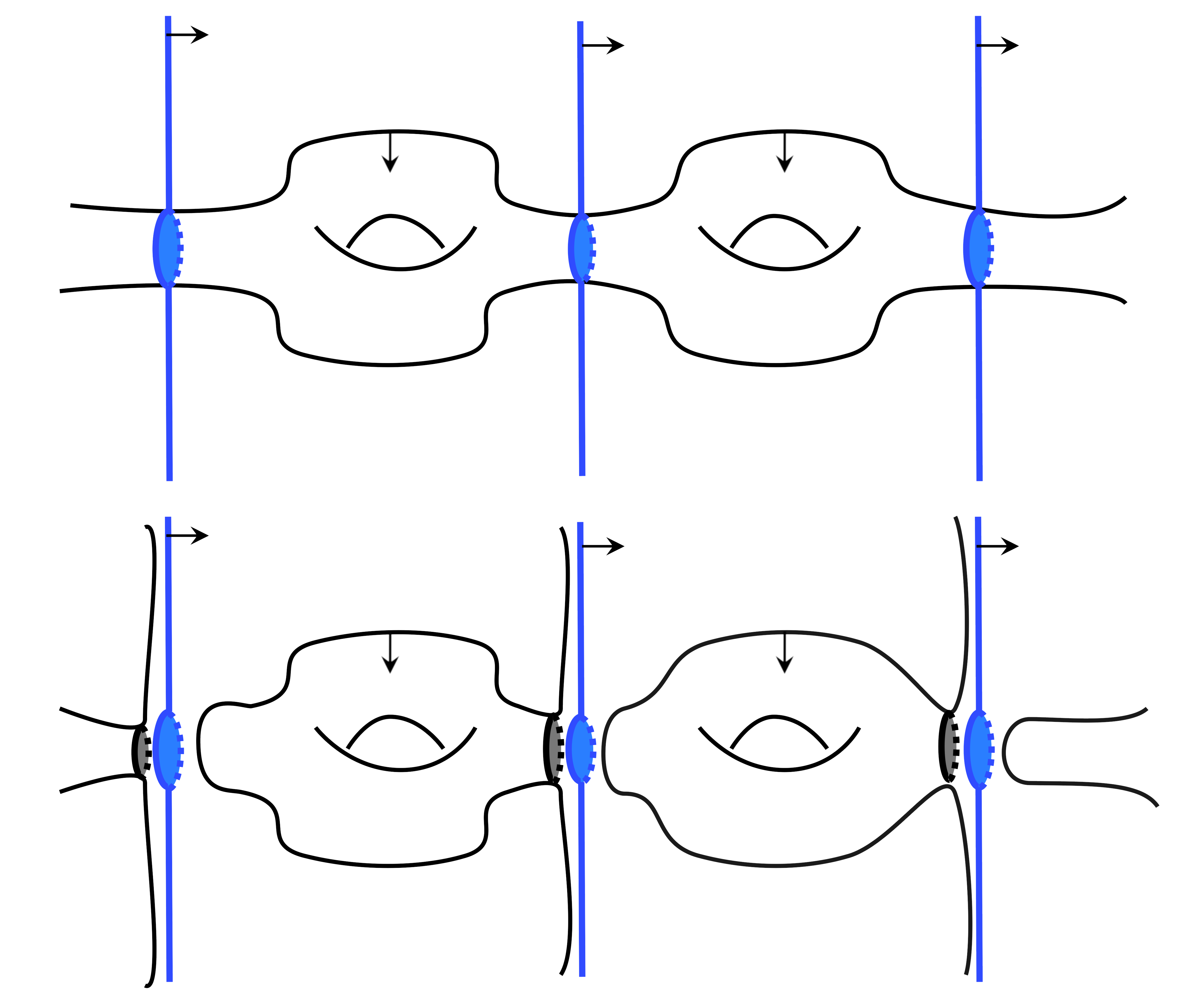}
\caption{The top figure shows a schematic depiction of the Heegaard surface $H$ (in black) and the incompressible surface $F$. For both surfaces we indicate transverse orientations. In the second figure, in black, we show the Haken sum $H+F$ of $H$ and $F$; this is the surface obtained by smoothing the intersections using the orientations. According to Theorem \ref{GCT}, after a small isotopy $(H+F) \cup F$ is a generalized Heegaard surface. }
\label{fig: HakenSum}
\end{figure}

\section{Outline}

Heegaard splitting theory has reached considerable maturity as a mathematical subject. A key insight that is the basis of many modern results is the realization (perhaps originating in \cites{G3, Kobayashi}) that strongly irreducible Heegaard splittings can isotoped to intersect essential surfaces in circles of intersection that are mutually essential or mutually inessential. This result does not use Haken's Lemma; we sketch the proof in Lemma \ref{SISHL}. Before expounding on some of the details, here is a brief sketch of the argument proving the Strong Haken Theorem, for the experts. For simplicity, assume that $H$ is a Heegaard surface for a closed, connected, orientable, reducible 3-manifold $M = M_1 \# M_2$ where both $M_1$ and $M_2$ are prime and irreducible (hence, not 3-balls). Let $S \subset M$ be an essential sphere. In this simplified setting $S$ is unique up to isotopy. (Consequently, Haken's Lemma will directly imply the Strong Haken Theorem, however to give an argument that applies more generally we won't use that fact.) Thin $H$ to a generalized Heegaard surface $\mc{H}$ consisting of thick surfaces $\mc{H}^+$ and thin surfaces $\mc{H}^-$. This can be done so that $\mc{H}^-$ is incompressible and each component of $\mc{H}^+$ is strongly irreducible. The incompressibility of $\mc{H}^-$ and the strong irreducibility of $\mc{H}^+$ can be used to isotope $\mc{H}$ to be disjoint from $S$. It then follows that some component of $\mc{H}^-$ is an essential sphere $P$. By our simplifications, the surface $S$ must be parallel to $P$ and surgery on $P$ produces $M_1$ and $M_2$. On each side of $P$ amalgamate $\mc{H}$ to Heegaard surfaces $H_1$ and $H_2$ for $M_1$ and $M_2$. We may then amalgamate $H_1 \cup P \cup H_2$ to a Heegaard surface $J$. We perform this amalgamation carefully. Start by sliding the feet of the 1-handles,  of the compressionbody between $P$ and $H_1$ so that when we extend them to lie on $H_2$, we need only extend a single foot. This ensures that $P$ (and, in this simplified setting, $S$) intersect $J$ in a single simple closed curve. We may then appeal to a result of Lackenby \cite{Lackenby}*{Prop. 3.1} which guarantees that $H$ is isotopic to $J$. In the general setting, where $M$ admits multiple essential spheres and discs, and especially when some of those spheres may be nonseparating, we need to do more work to implement the concept.

Here is a bit more information on the overall strategy. Thin position for 3-manifolds is technique for acquiring topological information from a weakly reducible Heegaard splitting of a 3-manifold $M$. It involves replacing a Heegaard surface with a (likely disconnected surface) called a \emph{generalized Heegaard surface}; each component of a generalized Heegaard surface is either a \emph{thick surface} or a \emph{thin surface}. The thick surfaces are Heegaard surfaces for the regions of $M$ complementary to the thin surfaces.  The basic idea (originating with Casson and Gordon \cite{CG}) is to decompose a weakly reducible Heegaard surface into three mutually disjoint surfaces (each is potentially disconnected), two of which are thick and one of which is thin.  Thinking of a Heegaard splitting as a sequence of handle attachments, where 0-handles and 1-handles are attached before 2-handles and 3-handles, this operation (called ``weak reduction'') is equivalent to attaching a 2-handle (and possibly 3-handle) prior to attaching a 1-handle (and possibly a 0-handle). If one or both of the thick surfaces are themselves weakly reducible we can repeat the operation. The end result is a collection of thick and thin surfaces. If it is ever the case that a thick and thin surface are parallel, we remove them from the collection. This way of creating and removing thick and thin surfaces is called \defn{thinning}.  Scharlemann and Thompson \cite{ST} produce a complexity that decreases under thinning; a collection of thin and thick surfaces is \emph{thin} if it minimizes their complexity. In a thin collection, the thin surfaces are incompressible and the thick surfaces are strongly irreducible. (In \cite[Rule 5]{ST}, the authors appeal to Lemma \ref{HL-disc} for this fact; we replace it with Lemma \ref{SISHL}.) Scharlemann and Thompson observe that their complexity is (in a certain sense) additive under connected sum. Although they do not indicate what they had in mind for a proof, it is likely they intended to use Haken's Lemma, together with the prime decomposition theorem for 3-manifolds. Our route on the other hand, proves Haken's Lemma and the Strong Haken Theorem using (in a certain sense) additivity.  

We prove Haken's Lemma in Theorem \ref{Haken Lem} using thin position machinery and its inverse, amalgamation, a technique due to Schultens \cite{Schultens}. Amalgamation converts a generalized Heegaard surface back into a Heegaard surface. In principle, there are innumerable different ways of amalgamating a generalized Heegaard surface. However, the aforementioned result of Lackenby will show that the results are all isotopic. This allows us, in Section \ref{amalg}, to make certain choices during the amalgamating process to control the intersection with thin spheres. 

Also, in Section \ref{amalg} we prove the Structure Theorem for Weakly Reducible Heegaard Surfaces (Theorem \ref{GCT}).

Our proof of the Strong Haken Theorem requires one additional technique, we call \emph{juggling}: it allows thin spheres to be moved to different places with respect to the other components of a generalized Heegaard surface. We explain it in Section \ref{juggling}. 

The proof of the Strong Haken Theorem is then carried out at the end of Section \ref{juggling} and in Section \ref{strong haken} as follows. We are given a Heegaard surface $H$ for a 3-manifold that contains properly embedded essential spheres and discs. We are also given a mutually disjoint collection $\mc{S}$ of essential spheres and discs. We thin $H$ to arrive at a thin generalized Heegaard surface $\mc{H}$. We use juggling to ensure it is disjoint from $\mc{S}$. We then carefully amalgamate $\mc{H}$ to a Heegaard surface $J$ and appeal to a version of Lackenby's result (Corollary \ref{Lack thm}) to know that $H$ is isotopic to $J$. 

\subsection{A word on notation}
Unless otherwise specified all 3-manifolds $M$ are compact, connected, and orientable. All surfaces are compact and orientable; if they are said to lie in a 3-manifold they are usually assumed to be properly embedded. In a directed graph (aka \defn{digraph}) a \defn{coherent cycle} (resp. \defn{coherent path}) is an embedded oriented loop (resp. path) such that the orientation of the loop (resp. path) coincides with the orientation of each edge it traverses. As in graph theory, a digraph is \defn{acyclic} if it does not contain a coherent cycle; the underlying undirected graph of an acyclic digraph may contain cycles. We let $I = [0,1]$.  For a (properly embedded) surface $S$ in a 3-manifold $M$, we let $M\setminus S$ denote the result of removing a regular neighborhood of $S$ from $M$. If $D$ is the union of pairwise disjoint discs properly embedded in $M$, we let $M|_D$ denote the result of $\boundary$-reducing $M$ along $D$. The \defn{scars} from $D$ are the discs in $\boundary M|_D$ that are the feet of the 1-handles with cocores $D$ removed from $M$. Similarly, if $S \subset M$ is a sphere, we let $M|_S$ denote the result of surgery along $S$, that is the result of removing a regular neighborhood of $S$ from $M$ and then gluing in 3-balls to the resulting boundary components. In $M\setminus S$, the \defn{scars} from $S$ are the feet of the copies of $S^2 \times I$ removed from $M$. In $M|_S$, the \defn{scars} from $S$ are the 3-balls that were glued to the scars in $M\setminus S$.

\section{Compressionbodies and Generalized Heegaard Surfaces.}

A \defn{compressing disc} for a surface $S$ is an embedded disc $D \subset M$  such that:
\begin{enumerate}
\item the simple closed curve $\boundary D$ is contained in $S$, 
\item the interior of $D$ is disjoint from $S$, and 
\item the curve $\boundary D$ does not bound a disc in $S$. 
\end{enumerate}
An \defn{s-disc} is a generalization of ``compressing disc'': the definition is identical except we replace (3) with:
\begin{enumerate}
\item[(3')] There is no disc $E \subset S$ with $\boundary D = \boundary E$ such that $D \cup E$ bounds a ball in $M$ with interior disjoint from $S$.
\end{enumerate}
A \defn{semi-compressing disc} for $S$ is an s-disc that is not a compressing disc. It necessarily has inessential boundary in $S$. If $\Delta$ is the union of mutually disjoint s-discs for $S$, we can \defn{compress} $S$ to a (possibly disconnected) surface $S|_\Delta$ by letting $S|_\Delta$ be the result of removing a regular neighborhood $N(\Delta) \approx (\Delta \times I)$ of $\boundary \Delta$ from $S$ and then taking the union with $\Delta \times \boundary I$. Observe that $S|_\Delta$ inherits a normal orientation from $S$ and that each disc in $\Delta$ leaves two scars on $S|_\Delta$. 

A surface $S$ in a 3-manifold $M$ is:
\begin{itemize}
\item  \defn{compressible} if there is a compressing disc for $S$ and \defn{incompressible} if there is no such disc; 
\item \defn{weakly reducible} if there exist disjoint s-discs $D_1$ and $D_2$ for $S$, with boundaries on the same component of $S$, such that near $S$, the discs $D_1$ and $D_2$ lie on opposite sides of the component of $S$ containing their boundaries and \emph{at least one of} $D_1$ or $D_2$ is a compressing disc. (``Near $S$'' means that in a regular neighborhood $N(S)$ of $S$, which may be nonseparating in $M$, each $D_i \cap N(S)$ is a collar neighborhood of $\boundary D_i$, and these collar neighborhoods lie on opposite sides of $S$);
\item \defn{strongly irreducible} if it is not weakly reducible.
\end{itemize}

A \defn{ball compressionbody} $C$ is an oriented 3-manifold homeomorphic to a 3-ball. In this case, we let $\boundary_+ C = \boundary C$ and $\boundary_- C = \nil$. A \defn{product compressionbody} $C$ is an oriented 3-manifold homeomorphic (via an orientation-preserving homeomorphism) to $F \times I$ where $F$ is a closed, connected surface. We let $\boundary_+ C$ be the component of $\boundary C$ taken to $F \times \{1\}$ and $\boundary_-C$ be the component taken to $F \times \{0\}$. A \defn{trivial compressionbody} is either a ball compressionbody or a product compressionbody. A \defn{compressionbody} is a compact, connected, oriented 3-manifold $C$ with nonempty boundary, such that:
\begin{enumerate}
\item One component of $\boundary C$ is designated as $\boundary_+ C$ and the union of the other components is designated as $\boundary_- C$,
\item There is a (possibly empty) collection $\Delta$ of mutually disjoint s-discs for $\boundary_+ C$, such that $C|_\Delta$ is the union of trivial compressionbodies and the designations of $\boundary_+$ and $\boundary_-$ are inherited from $C$. Such a collection $\Delta$ is a \defn{complete collection of s-discs for $C$}.
\end{enumerate}
A \defn{punctured trivial compressionbody} is the result of removing the open regular neighborhood of a finite (possibly empty) collection of points from the interior of a trivial compressionbody. The positive boundary of a punctured product compressionbody is the same as the positive boundary prior to removing the open balls. Observe that a punctured product compressionbody is a compressionbody. 

A \defn{vertical arc} $\alpha$ in a compressionbody $C$ is a properly embedded arc such that there exists a complete collection $\Delta$ of s-discs for $C$ disjoint from $\alpha$ such that $\alpha$ is isotopic to a fiber in one of the product compressionbody components of $C|_\Delta$. 

A \defn{spine} $\Gamma$ for $C$ is the union of $\boundary_- C$ with a properly embedded 1-dimensional CW-complex in $C \setminus \boundary_+ C$ such that $C$ deformation retracts to $\Gamma$. A spine can be constructed by attaching to $\boundary_- C$ or to a vertex (if $\boundary_- C = \nil$) the cocores of the 2-handles that form a regular neighborhood of any complete collection of s-discs.

\begin{lemma}\label{subcompbodies}
The following are true for a compressionbody $C$:
\begin{enumerate}
\item $\boundary_- C$ is incompressible in $C$.
\item If $D \subset C$ is the union of mutually disjoint s-discs for $\boundary_+ C$, then there exists a complete collection $\Delta$ of s-discs for $C$  such that $D \subset \Delta$.
\item If $D \subset C$ is the union of mutually disjoint s-discs for $\boundary_+ C$, then $C|_D$ is the disjoint union of compressionbodies, with orientations and $\boundary_+$ and $\boundary_-$ designations inherited from $C$.
\item $C$ does not contain an s-disc for $\boundary_+ C$ if and only if it is a trivial compressionbody. If $C$ contains an s-disc for $\boundary_+ C$ but does not contain a compressing disc for $\boundary_+ C$, then it is a punctured product compressionbody.
\item\label{5} If $S \subset C$ is a closed connected incompressible surface in the interior of $C$, then $S$ is separating and separates $C$ into two compressionbodies. One of the compressionbodies has (a copy of) $S$ as its positive boundary and components of $\boundary_- C$ as negative boundary and the other has $\boundary_+ C$ as its positive boundary and (a copy of) $S$ and components of $\boundary_- C$ as its negative boundary. The compressionbody containing $S$ as its positive boundary has at most one negative boundary component that is not a sphere; it contains such a component $F$ if and only if $S$ is not a sphere. If $S$ is not a sphere that compressionbody is a punctured product compressionbody between $S$ and $F$.
\end{enumerate}
\end{lemma}

\begin{proof}
Properties (1) - (3) are well-known and easily proved. See \cite[Lemma 3.5]{TT1} for a generalization to the situation when $C$ contains a certain type of properly-embedded graph. The case when we are reducing along a semi-compressing disc requires the use of the lightbulb trick (see \cite{HW} for a description). Property (4) follows from the definition of compressionbody.

For the proof of (5), suppose that $S \subset C$ is a closed, connected, incompressible surface in the interior of $C$. Let $\Delta$ be a complete collection of s-discs for $C$ chosen to intersect $S$ minimally. An innermost disc argument then shows that $S$ intersects only semi-compressing discs in $\Delta$. Let $\Delta' \subset \Delta$ be the union of the compressing discs and let $W \subset C|_{\Delta'}$ be the component containing $S$. Let $\wihat{W}$ be the result of capping off spherical components of $\boundary_- W$ with 3-balls. Note that $\wihat{W}$ is either a 3-ball or a product compressionbody. This does not change the incompressibility of $S$. Hence, if $\wihat{W}$ is a 3-ball, then $S$ is a 2-sphere bounding a 3-ball in $\wihat{W}$, and the result follows easily. Suppose that $\wihat{W}$ is a product compressionbody. Waldhausen's classification of surfaces in product manifolds \cite{Waldhausen-sufflarge} shows that either $S$ is parallel in $\wihat{W}$ to the unique component of $\boundary_- \wihat{W}$ or is a sphere bounding a 3-ball in $\wihat{W}$. Conclusion (5) follows by reconstituting $C$ from $\wihat{W}$.
\end{proof}

\begin{remark}
As in Figure \ref{semi-discs for negative boundary}, $\boundary_- C$ may admit a semi-compressing disc. This observation will necessitate some extra work in what follows. However, Lemma \ref{thin s-discs} shows that such discs still cut off compressionbodies.
\end{remark}

\begin{figure}[ht!]
\labellist
\small\hair 2pt
\pinlabel {$C$} at 166 120
\pinlabel {$\boundary_+ C$} [r] at 20 211
\pinlabel {$\boundary_- C$} [r] at 97 34
\pinlabel {$\boundary_- C$} [t] at 316 87
\pinlabel {$D$} [bl] at 375 150
\endlabellist
\centering
\includegraphics[scale=0.35]{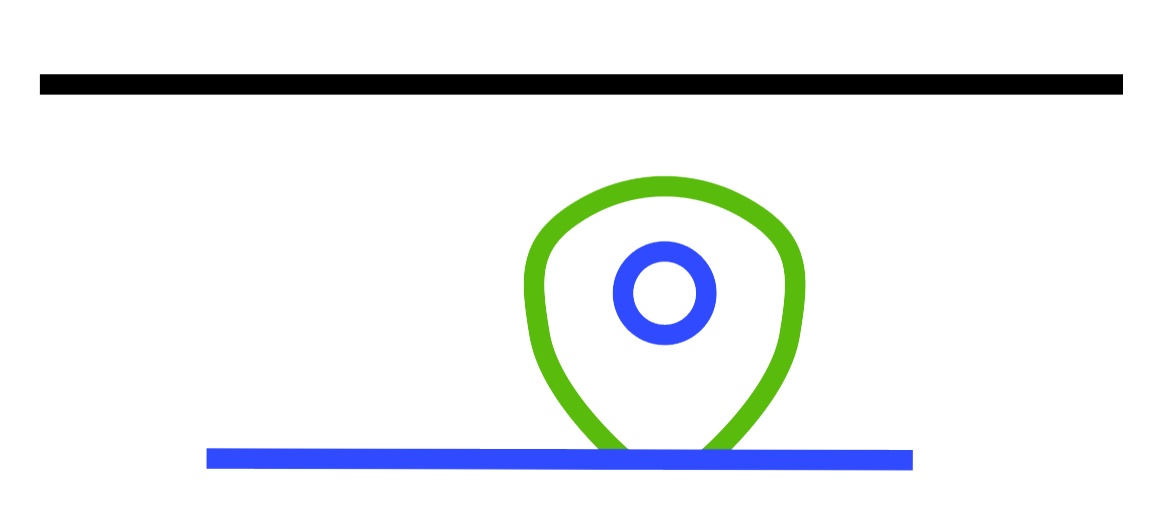}
\caption{It is possible for the negative boundary of a compressionbody $C$ to admit a semi-compressing disc $D$.}
\label{semi-discs for negative boundary}
\end{figure}

\begin{lemma}\label{thin s-discs}
Suppose that $C$ is a compressionbody and that $D$ is a semi-compressing disc for $\boundary_- C$. Then, together with a disc $E \subset \boundary_- C$, $D$ bounds a compressionbody in $C$ disjoint from $\boundary_+ C$. This compressionbody is the result of removing at least one open 3-ball from a 3-ball with boundary $D \cup E$.
\end{lemma}
\begin{proof}
By definition of semi-compressing disc, $\boundary D$ bounds at least one disc in $\boundary_- C$. As every sphere in a compressionbody is separating and since $\boundary_- C$ is incompressible in $C$, there is such a disc $E$ such that the sphere $D \cup E$ bounds a submanifold $W\subset C$ disjoint from $\boundary_- C$. The result follows as in the proof of Lemma \ref{subcompbodies} (\ref{5}).
\end{proof}

The next observation concerning handles is key to the process of amalgamating a generalized Heegaard surface so that it intersects a sphere exactly once. See Figure \ref{handleslide} for an example.
\begin{lemma}[Monopod Lemma]\label{monopod lemma}
Given any compressionbody $C$ with $\boundary_- C \neq \nil$, there exists a collection of s-discs $D \subset C$ such that each component of $C|_D$ is either a handlebody or a product compressionbody and each component contains exactly one scar from $D$. Furthermore, if $A \subset C$ is the union of finitely many pairwise disjoint vertical annuli and compressing discs for $\boundary_+ C$, we may choose $D$ to be disjoint from $A$.
\end{lemma}
\begin{proof}
The compressionbody $C$ is obtained from $\boundary_- C \times I$ by attaching 1-handles to $\boundary_- C \times \{1\}$. We may slide the ends of the 1-handles onto other 1-handles to ensure that each component of $\boundary_- C \times \{1\}$ is incident to exactly one foot of the 1-handles. Given $A$, we may start with a choice of 1-handles whose co-cores are disjoint from the annuli components of $A$ and parallel to the disc components of $A$. The slides may then be performed so that the resulting discs $D$ are disjoint from $A$.
\end{proof}

\begin{figure}[ht!]
\labellist
\small\hair 2pt
\pinlabel {$\boundary_+ C$} [r] at 15 79
\pinlabel {$\boundary_- C$} [r] at 15 25
\endlabellist
\centering
\includegraphics[scale=0.3]{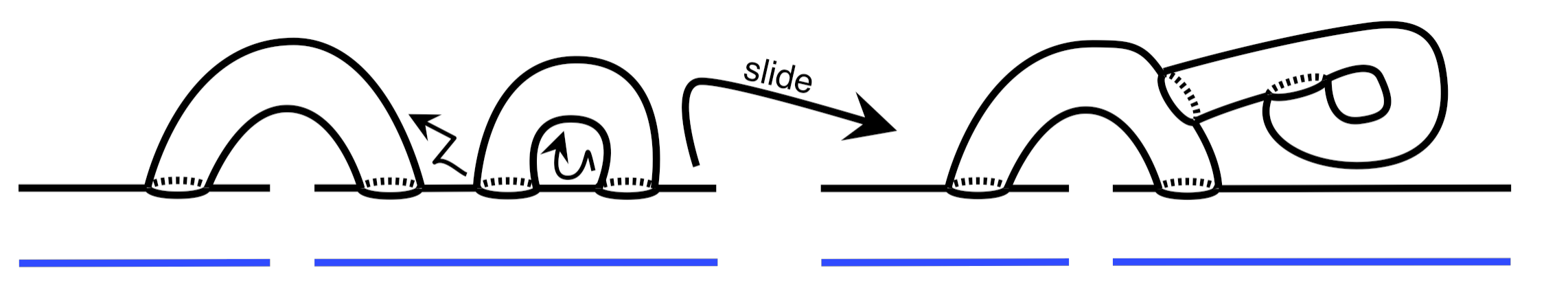}
\caption{By sliding 1-handles incident to $\boundary_+ C$, we may find discs separating a compressionbody into product compressionbodies and handlebodies so that each product compressionbody contains a single scar.}
\label{handleslide}
\end{figure}

A \defn{generalized Heegaard surface} $\mc{H}$ for  a 3-manifold $M$ is a closed properly embedded surface such that the following hold:
\begin{enumerate}
\item $\mc{H} = \mc{H}^+ \cup \mc{H}^-$, where each of $\mc{H}^+$ and $\mc{H}^-$ are unions of components of $\mc{H}$ and $\mc{H}^+ \cap \mc{H}^- = \nil$.
\item The closure of the complement of $\mc{H}$ in $M$ (denoted $M\setminus \mc{H}$) is the union of compressionbodies. 
\item $\mc{H}^+ = \bigcup\limits_C \boundary_+ C$
\item $\boundary M \cup \mc{H}^- = \bigcup\limits_C \boundary_- C$
\end{enumerate} 
where the unions appearing in the last two items are over all the components of $M \setminus \mc{H}$ and the equality holds after gluing. A generalized Heegaard surface $\mc{H}$ is \defn{oriented} if each component of $\mc{H} \cup \boundary M$ is assigned a transverse orientation such that for each compressionbody $C \subset M \setminus \mc{H}$, $C$ is a cobordism from $\boundary_- C$ to $\boundary_+ C$ or vice-versa. We say that a compressionbody $C$ is \defn{above} (resp. \defn{below}) $\boundary_+ C$ if the transverse orientation for $\boundary_+ C$ points into (resp. out of) $C$. Two oriented generalized Heegaard surfaces are equivalent if they are ambient isotopic in $M$, via an orientation-preserving isotopy. We will only work with oriented generalized Heegaard surfaces in what follows.

\begin{definition}\label{dual digraph}
Associated to each oriented generalized Heegaard surface $\mc{H}$ is a \defn{dual digraph}. This is a directed graph with a vertex for each component of $M\setminus \mc{H}$ and each directed edge from one vertex to another corresponding to an (oriented) component of $\mc{H}$. Some edges correspond to thick surfaces and some to thin surfaces. At each vertex, either there is one incoming thick edge and all thin edges are outgoing or there is one outgoing thick edge and all thin edges are incoming. In particular, the requirement that each compressionbody is a cobordism between $\boundary_- C$ and $\boundary_+ C$ ensures that no edge is a loop. We say that $\mc{H}$ is \defn{acyclic} if its dual digraph is acyclic. (These dual digraphs are essentially the \emph{fork complexes} of \cite{SSS}.) The dual digraph for an acyclic $\mc{H}$ may have underlying undirected graph with cycles.
\end{definition}

A generalized Heegaard surface is a Heegaard surface if and only if it is connected. In which case, $\mc{H} = \mc{H}^+$ and $\mc{H}^- = \nil$. 

\begin{lemma}\label{SISHL}
Suppose that $H$ is a strongly irreducible Heegaard surface for $M$ and that $\mc{S}$ is a nonempty collection of mutually disjoint essential spheres and discs. Then $H$ can be isotoped so that for each sphere $S_0 \subset \mc{S}$, $H \cap S_0$ is the union of loops that are inessential in $H$ and for each disc $S_0 \subset \mc{S}$, $H \cap S_0$ is the union of loops, at most one of which is essential in $H$. Furthermore, any component of $H \cap \mc{S}$ which is innermost in $\mc{S}$ bounds, in $\mc{S}$, an s-disc for $H$.
\end{lemma}

\begin{proof}
This is a by now standard sweepout argument so we simply sketch the proof. For more details, see (for example) \cite{RS}. Let $\Gamma_-$ and $\Gamma_+$ be spines for the compressionbodies on either side of $H$, transverse to $\mc{S}$. Their complement is a product which allows us to define a continuous function $\phi \co M \to I$. We may perturb it so $\phi|_{\mc{S}}$ is Morse. Note that $\phi^{-1}(0) = \Gamma_-$, $\phi^{-1}(1) = \Gamma_+$ and $H_t = \phi^{-1}(t)$ is isotopic to $H$ for any $t \in I\setminus \boundary I$. The critical values of $\phi$ divide $I$ into open sub-intervals. Label a subinterval with $-$ (resp. $+$) if some (equivalently, any) $t$ in the subinterval has the property that there is a loop in $H_t \cap \mc{S}$ that is essential in $H_t$ and which bounds a disc below $H_t$ (resp. above $H_t$). As $H$ (and, hence, $H_t$) is strongly irreducible no interval is labelled with opposite signs. Since both $H$ and $\mc{S}$ are orientable, it is also the case that no two adjacent intervals have opposite sign. 

Consider either the lowest interval or the highest interval. Let $t$ be in the interval and let $C$ be the compressionbody below or above $H_t$ respectively. We think of $H_t$ as being the frontier of a regular neighborhood of the spine. Let $S_0 \subset \mc{S}$ be a component. Each intersection of $S_0$ with $H$ is either a loop in $C$ parallel to $\boundary S_0$ or a meridian of an edge of the spine. A meridian intersection bounds an s-disc for $H$ in $C$. If some such disc is a compressing disc, the interval is labelled $-$ (if it is the lowest interval) or $+$ (if it is the highest interval). If not, $H_t$ is the surface isotopic to $H$ satisfying the conclusions of the lemma. 

Assume, therefore, that both sides of $H$ admit compressing discs, so the lowest interval is labelled $-$ and the highest is labelled $+$. We conclude that there must be a subinterval with no label. Let $t$ lie in that subinterval. If $H_t \cap \mc{S} = \nil$, we are done, so suppose that $\gamma \subset H_t \cap \mc{S}$ is a component. As $\mc{S}$ is the union of discs and spheres, we may assume that it is innermost in $\mc{S}$. Let $D \subset \mc{S}$ be the innermost disc it bounds. Since the interval containing $t$ has no label, $\boundary D$ is inessential on $H$. If $D$ is not a semi-compressing disc for $H$, we can isotope $H$ so as to remove $\boundary D$ from $H \cap \mc{S}$.
\end{proof}

\section{Heegaard surfaces through thick and thin}

Following \cite{HS}, we reinterpret \cite{ST} in terms of surfaces, rather than handle structures. Suppose that $\mc{H}$ is an oriented generalized Heegaard surface and that $H$ is an weakly reducible thick surface. Let $D_-$ and $D_+$ be unions of pairwise disjoint s-discs for $H$, with $D_-$ below $H$ and $D_+$ above $H$. Let $H_\pm = H|_{D_\pm}$ and let $F = H|_{D_- \cup D_+}$. A small isotopy of $H_+$ above $H$ and $H_-$ below $H$ makes $H_-$, $H_+$, and $F$ mutually disjoint. Let $\mc{J}^+ = (\mc{H} \setminus H) \cup H_- \cup H_+$ and let $\mc{J}^- = \mc{H}^- \cup F$. Then we say that $\mc{J} = \mc{J}^- \cup \mc{J}^+$ is obtained by an \defn{untelescoping} $\mc{H}$ and that $D_-$ and $D_+$ are an \defn{s-weak reducing pair}. It is a \defn{weak reducing pair} if both $D_-$ and $D_+$ contain a compressing disc and a \defn{(s/2)-weak reducing pair} if at least one of $D_-$ or $D_+$ contain a compressing disc. See Figure \ref{fig:untel}.

\begin{figure}[ht!]
\labellist
\small\hair 2pt
\pinlabel {$H$} [r] at 20 80
\pinlabel {$D_-$} [t] at 143 21
\pinlabel {$D_+$} [b] at 63 142
\pinlabel {$H_+$} [l] at 524 110
\pinlabel {$H_-$} [l] at 524 52
\pinlabel {$F$} [l] at 524 83
\endlabellist
\centering
\includegraphics[scale=0.4]{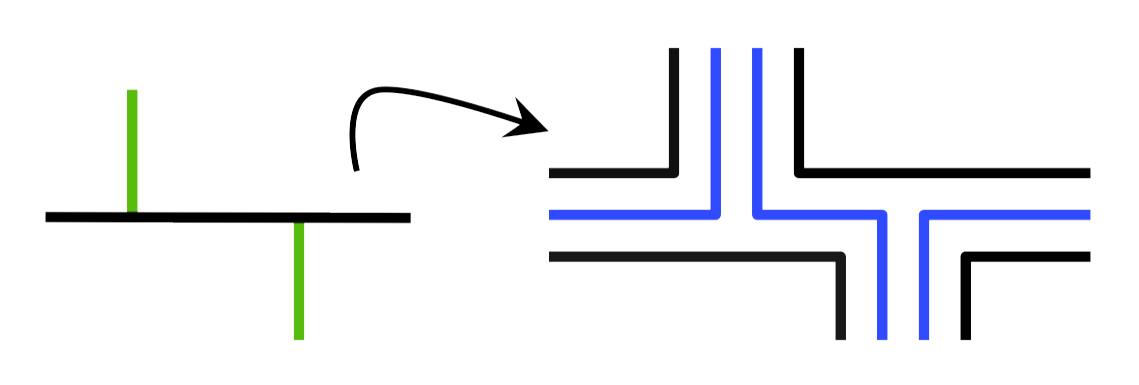}
\caption{Untelescoping a thick surface $H$ using a pair of s-discs $D_-$ and $D_+$.}
\label{fig:untel}
\end{figure}

The following theorem is essentially due to Scharlemann-Thompson, however they do not address the situation with semi-compressing discs; that is explored in depth in \cite{TT1}. In this paper, we somewhat modify both approaches, for simplicity.

\begin{theorem}\label{untel}
If $\mc{H}$ is an oriented generalized Heegaard surface and $\mc{J}$ is obtained from $\mc{H}$ by untelescoping, then $\mc{J}$ is an oriented generalized Heegaard surface. Furthermore, if $\mc{H}$ is acyclic, so is $\mc{J}$.
\end{theorem}
\begin{proof}
We leave the proof to the reader; see \cites{ST, HS, TT1} for example. 
\end{proof}

On the other hand, suppose $\mc{H}$ contains a thick surface $H$ and a thin surface $F$ such that $H \cup F$ lie in the boundary of a product compressionbody component $C$ of $M\setminus \mc{H}$. We say that $\mc{J} = \mc{H}\setminus (H \cup F)$ is obtained by \defn{consolidating} $\mc{H}$.  

\begin{lemma}[{Taylor-Tomova \cite{TT1}}]\label{consol preserve acyclic}
Suppose $\mc{H}$ is an oriented generalized Heegaard surface and that $\mc{J}$ is obtained from $\mc{H}$ by a consolidation. Then $\mc{J}$ is an oriented generalized Heegaard surface. Furthermore, if $\mc{H}$ is acyclic, so is $\mc{J}$.
\end{lemma}

We identify three \defn{thinning moves} that may be able to be performed on an oriented generalized Heegaard surface $\mc{H}$. 

\begin{enumerate}
\item consolidation
\item untelescoping using a weak reducing pair
\item untelescoping using an (s/2)-weak reducing pair $D_-$, $D_+$, such that $D_\pm$ is a compressing disc and a positive genus component of $H_\mp = H|_{D_\mp}$ can be consolidated with a component of $\mc{H}^-$; we also perform that consolidation.
\end{enumerate}
We call (2) a \defn{type I untelescoping} and (3) a \defn{type II untelescoping and consolidation}

A \defn{thinning sequence} is a sequence of consolidations, type I untelescopings, and type II untelescoping and consolidations.  If $\mc{J}$ is obtained from $\mc{H}$ by a (possibly empty) thinning sequence, we say that $\mc{H}$ \defn{thins} to $\mc{J}$. If none of the three thinning moves is possible for $\mc{H}$, then we say $\mc{H}$ is \defn{locally thin}. 

\begin{remark}
In many papers, including \cites{HS, ST, TT1}, additional moves are also allowed, corresponding to the destabilization of Heegaard surfaces. We do not need to make use of those moves. 
\end{remark}

\begin{theorem}\label{locally thin}
Suppose that $\mc{H}$ is a generalized Heegaard surface for a 3-manifold $M$. Then every thinning sequence beginning with $\mc{H}$ must terminate. Consequently, there exists a locally thin $\mc{J}$ such that $\mc{H}$ thins to $\mc{J}$. If $\mc{H}$ is acyclic, so is $\mc{J}$.
\end{theorem}
\begin{proof}
We have already observed in Lemmas \ref{untel} and \ref{consol preserve acyclic} that if an oriented $\mc{H}$ thins to $\mc{J}$, then $\mc{J}$ is oriented  and if $\mc{H}$ is acyclic so is $\mc{J}$. We need only show that every thinning sequence terminates. Scharlemann and Thompson define the \defn{width} of a generalized Heegaard splitting. In our context, this is a complexity for a generalized Heegaard surface. Adapting their definition slightly, we let $c(\mc{H})$ be the sequence whose entries are $c(H) = (1 -\chi(H)/2)$ for each component $H \subset \mc{H}^+$, arranged in non-increasing order. We note that each entry of the sequence $c(\mc{H})$ is a non-negative integer that is 0 if and only if the corresponding surface is a sphere.  If $D$ is a collection of pairwise disjoint s-discs for $H$, then each component $H' \subset H|_D$ has $c(H') \leq c(H)$. If $D$ contains at least one compressing disc, then the inequality is strict for every component $H'$. If $D$ contains only s-discs, then there is one component $H'$ of the same complexity as $H$ and all others are spheres and have complexity zero.

If $\alpha, \beta$ are two such complexity sequences, we say that $\alpha < \beta$ if and only if there exists $k \geq 0$ such that the first $k$ terms of $\alpha$ and $\beta$ coincide and either $\alpha$ does not have a $(k+1)$st term, but $\beta$ does (possibly it is 0) or the $(k+1)$st term of $\beta$ is strictly greater than the $(k+1)$st term of $\alpha$. Observe that complexities are well-ordered. 

Consolidation removes a term from a complexity sequence and so decreases it. A type I untelescoping replaces an entry in a complexity sequence with some number (at least two) of other terms, each strictly smaller than the original. Thus, it also decreases complexity. Finally, consider a type II untelescoping and consolidation. Assume the untelescoping uses an (s/2)-weak reducing pair $D_-$, $D_+$ such that $D_-$ contains only semi-compressing discs and $D_+$ contains at least one compressing disc. Let $H \subset \mc{H}^+$ contain $\boundary D_- \cup \boundary D_+$. Since $H$ admits a compressing disc, $c(H) > 0$. Then if $H'$ is a component of $H_+ = H|_{D_+}$, we have $c(H') < c(H)$. If $H'$ is a component of $H_- = H|_{D_-}$, either $c(H') = c(H)$ or $c(H') = 0$. In the former case, we immediately consolidate the component $H'$, and so the type II untelescoping and consolidation also strictly decreases complexity.
\end{proof}

\begin{lemma}\label{thin props}
Suppose that $\mc{H}$ is a locally thin oriented generalized Heegaard surface for $M$. Then the following hold:
\begin{enumerate}
\item Each component of $\mc{H}^+$ is strongly irreducible;
\item No component of $M \setminus \mc{H}$ is a product compressionbody between components of $\mc{H}^-$ and $\mc{H}^+$;
\item If a component $C \subset M\setminus \mc{H}$ is a punctured product compressionbody with $\boundary_+ C$ of positive genus and the positive genus component of $\boundary_- C$ in $\mc{H}^-$, then the compressionbody on the other side of $\boundary_+ C$ is a punctured product compressionbody;
\item Each component of $\mc{H}^-$ is incompressible in $M$.
\end{enumerate}
\end{lemma}

\begin{proof}
Each component of $\mc{H}^+$ is strongly irreducible in $M \setminus \mc{H}^-$ as otherwise we could further untelescope $\mc{H}$. 

Suppose that $C \subset M\setminus \mc{H}$ is a punctured product compressionbody with $H = \boundary_+ C$ of positive genus and $F \subset \boundary_- C \cap \mc{H}^-$ also of positive genus. $C$ cannot be a product compressionbody for then $\mc{H}$ would not be locally thin. Thus $\boundary_- C$ also contains a sphere and, therefore, a semi-compressing disc. Let $D \subset M\setminus \mc{H}$ be the compressionbody on the other side of $H$ from $C$. If $D$ contains a compressing disc for $H$, then it could be isotoped to be disjoint from any semi-compressing disc in $C$. Consequently, we can perform a type II untelescoping and consolidation move. This contradicts the assumption that $\mc{H}$ is locally thin. Thus $D$ does not contain a compressing disc for $H$. It must, therefore, be a punctured product compressionbody. 

Suppose a component $F$ of $\mc{H}^-$ is compressible in $M$ via a compressing disc $D$. Choose $F$ and $D$ so that $D$ intersects $F \subset \mc{H}^-$ minimally in its interior. An innermost disc argument shows that the interior of $D$ is disjoint from $\mc{H}^-$. Let $W$ be the component of $M\setminus \mc{H}^-$ containing $D$ and let $H$ be the component of $\mc{H}^+$ in $W$. By Lemma \ref{SISHL}, we may isotope $H$ so that $H \cap D$ consists of a single curve that is essential in $H$ and, possibly, a collection of inessential curves. Surgery on $D$ removes those inessential curves, so we may arrange that $D \cap H$ is a single essential curve in $H$. Thus, on one side $C$ of $H$, $D\setminus H$ is a vertical annulus, and on the other side it is a compressing disc for $H$. If $C$ contained a compressing disc for $H$, there would be one disjoint from $D$ and we would contradict the fact that $H$ is strongly irreducible. Thus, $C$ is a punctured product compressionbody. Note that $\boundary_+ C$ has positive genus. By (2), the compressionbody on the opposite side of $H$ is also a punctured compressionbody; however, punctured compressionbodies do not admit compressing discs for their boundaries. Thus, $\mc{H}^-$ is incompressible.
\end{proof}

\begin{remark}
If $\mc{H}$ is locally thin, we cannot guarantee that each component of $\mc{H}$ is essential; there may be components that are 2-spheres bounding 3-balls or components that are $\boundary$-parallel. Such thin surfaces will not interfere in what follows. As in \cites{ST, HS, TT1}, if we allowed certain destabilization operations, we could continue thinning and ensure the thin surfaces were essential. In this paper, we don't do this because we want to be able to reconstruct (via amalgamation) a Heegaard surface $H$ from a locally thin generalized Heegaard surface $\mc{H}$ to which it thins.
\end{remark}

\begin{corollary}\label{thin spheres exist}
Assume that $M$ is a reducible manifold such that no component of $\boundary M$ is a sphere. If $\mc{H}$ is a locally thin generalized Heegaard surface for $M$, then $\mc{H}^-$ contains a collection of essential spheres $P$ such that each component of $M|_P$ is irreducible. In particular, if $M$ contains a nonseparating sphere, then so does $\mc{H}^-$.
\end{corollary}
\begin{proof}
Let $P \subset \mc{H}^-$ be the union of all the spheres in $\mc{H}^-$ that are essential in $M$. \emph{A priori} $P$ may be empty, however we will see that it is not. Let $M_0$ be a component of $M|_P$ and let $\mc{H}_0$ be those components of $\mc{H}$ contained in the interior of $M_0$. Using the fact that $\mc{H}$ cannot be consolidated and that the existence of spherical components in $\boundary_- C$ for a compressionbody $C$, does not affect which curves in $\boundary_+ C$ bound compressing discs in $C$, we see that $\mc{H}_0$ is also thin. 

Suppose that $M_0$ contains an essential sphere $Q$. Isotope it to intersect $\mc{H}_0$ minimally. An innermost disc argument shows that there is a component $M'_0$ of $M_0 \setminus \mc{H}^-_0$ that either contains $Q$ or contains a semi-compressing disc for a component of $\mc{H}^-_0$. In the latter case, by Lemma \ref{thin s-discs}, the compressionbody containing that disc must have a sphere in its negative boundary. That sphere cannot lie in $\boundary M_0$ so it belongs to $\mc{H}^-_0$. But in that case, it is not essential so it bounds a 3-ball and we can further isotope $Q$ to remove the intersection. Thus, $Q \cap \mc{H}^-_0 = \nil$. Lemma \ref{SISHL} shows that this implies some component of $\mc{H}_0^+$ is weakly reducible, a contradiction. 

Since the previous paragraph modifies the sphere $Q$ by surgery, and since every sphere in a compressionbody is separating, if $Q$ were nonseparating, we see that $P$ must also have contained a nonseparating sphere.
\end{proof}

\section{Amalgamation}\label{amalg}

Suppose that $M$ is a 3-manifold containing a transversally oriented closed surface $F \subset M$ (possibly disconnected) separating $M$ into (possibly disconnected) 3-manifolds $M_-$ and $M_+$, such that the transverse orientation on $F$ points into each component of $M_+$ and out of each component of $M_-$ and so that each component of $M_-$ and $M_+$ is incident to a component of $F$. Suppose that $H_\pm$ is the union of oriented Heegaard surfaces for $M_\pm$ such that the orientation on $H_+$ (resp. $H_-$) points away from (resp. towards) $F$. We can \defn{amalgamate} the surfaces $H_-$ and $H_+$ across $F$ as follows. See Figure \ref{Fig: amalgamate}.

\begin{figure}[ht!]
\labellist
\small\hair 2pt
\pinlabel {$H_+$} [r] at 21 306
\pinlabel {$F$} [r] at 21 260
\pinlabel {$H_-$} [r] at 21 196
\pinlabel {$U_+$} at 62 352
\pinlabel {$U_-$} at 62 283
\pinlabel {$V_+$} at 62 230
\pinlabel {$V_-$} at 62 93
\pinlabel {$A$} [r] at 597 311
\pinlabel {$A$} [r] at 753 311
\pinlabel {$A$} [r] at 860 311
\pinlabel {$H$} [r] at 524 197
\endlabellist
\includegraphics[scale=0.4]{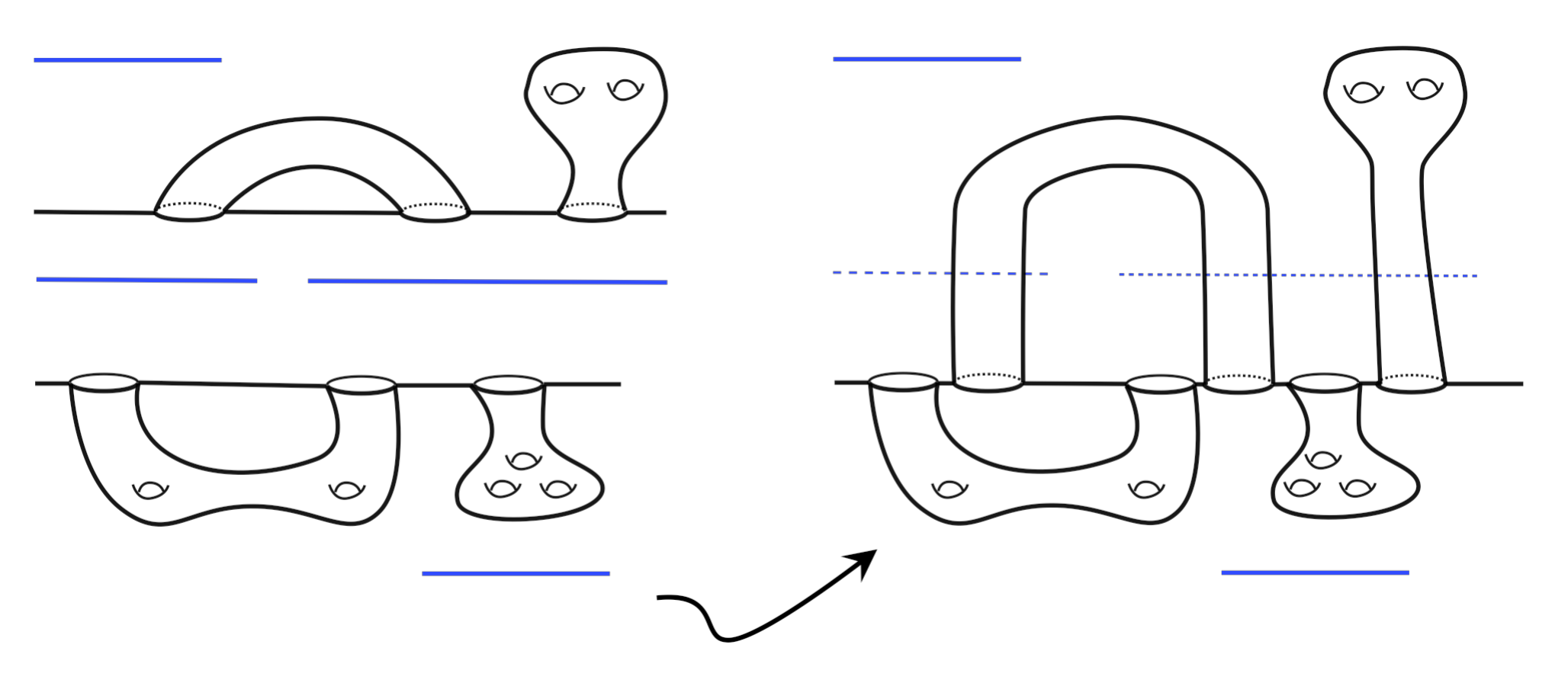}
\caption{The amalgamation of Heegaard surfaces $H_-$ and $H_+$ across a thin surface $F$.}
\label{Fig: amalgamate}
\end{figure}

For notational convenience, let $U_\pm, V_\pm$ be the union of the compressionbodies on either side of $H_\pm$ with $F = U_- \cap V_+$. Choose a complete collection of discs for each of $U_-$ and $V_+$. When we $\boundary$-reduce $U_- \cup V_+$ along those discs, we obtain product compressionbodies and 3-balls. The union of some of those product compressionbodies is $\Pi = F \times [-1,1]$ where $F \times [-1,0] \subset V_+$ and $F \times [0, 1] \subset U_-$. In both cases, consider $F$ as $F \times \{0\}$. The other product compressionbodies contain $\boundary_- V_+ \setminus F$ and $\boundary_- U_- \setminus F$. Let $\Delta_- \subset U_-$ be those discs in the complete collection that leave scars on $F \times \{1\}$. A small isotopy of the product structure on $\Pi$ ensures that the projection of the scars from $\Delta_-$ on $F \times \{1\}$ to $F \times \{-1\}$ completely misses the scars in $F \times \{-1\}$ arising from the complete collection of discs in $V_+$. Given a complete collection of discs for $V_-$, we may also ensure that the projections miss the remnants of their boundaries, as they are arcs in $F \times \{-1\}$.  Let $A \subset \Pi$ be the solid tubes corresponding to the projections. Let $X_-$ be the union of $V_-$ with $A$ and with the components of $U_-|_{\Delta_-}$ not containing $F$. Let $X_+$ be the closure of the complement of $X_-$ and let $H = X_- \cap X_+$. See Figure \ref{Fig: amalgamate}.

\begin{lemma}[{Schultens \cite{Schultens}}]\label{amalglem}
$H$ is a Heegaard surface for $M$.
\end{lemma} 
\begin{proof}
We observe that $X_-$ is the result of attaching 1-handles to $\boundary_+ V_-$ and so it is a compressionbody. From our construction, we note that the complete set of discs for $U_-$ together with a complete set of discs for $V_-$ are a complete set of discs for $X_-$.

On the other hand, consider a complete set of discs $\Delta_+$ for $U_+$, with boundaries transverse to $\Delta_-$. Then $\boundary \Delta_+ \cap F \times \{1\}$ is a collection of arcs and circles. We may use the product structure on $\Pi$ to extend $\boundary \Delta_+$ so that it lies on $H$ and is disjoint from the scars on $F \times \{-1\}$ left by the complete collection of discs for $V_+$. The union of the complete collection of discs for $V_+$ with $\Delta_+$ is then a complete collection of discs for $X_+$. Consequently, $X_+$ is a compressionbody.
\end{proof}

Two observations will be useful later:
\begin{lemma}\label{amalg prod}
Suppose that $M_-$ and $M_+$  are 3-manifolds such that $F = \boundary M_1 \cap \boundary M_2$ is nonempty and oriented so that its orientation points into $M_+$ and out of $M_-$. Let $H_-$ and $H_+$ be oriented Heegaard surfaces for $M_-$ and $M_+$ respectively such that the orientation for $H_+$ points away from $F$ and the orientation for $H_-$ points toward $F$. Let $H$ be the Heegaard surface resulting from amalgamating $H_+$ and $H_-$ across $F$. Then:
\begin{enumerate}
\item If $F$ and $H_\pm$ bound a product compressionbody in $M_\pm$, then (up to isotopy) $H = H_\mp$. That is, the amalgamation is the same as consolidation.
\item If $F$ is a sphere, then $H$ is isotopic to the connected sum of $H_-$ and $H_+$ as Heegaard surfaces for the 3-manifolds obtained from $M_-$ and $M_+$ by capping off $F$ with 3-balls.
\end{enumerate}
\end{lemma}
\begin{proof}
We leave the proof to the reader, although see the proof of the Strong Haken Theorem below for an idea that may help with (2).
\end{proof}

\begin{definition}\label{amalg def2}
Suppose that $\mc{H}$ is an oriented generalized Heegaard surface and that $H$ is a thick surface. Let $U_-$ (resp. $U_+$) be the compressionbody of $M\setminus \mc{H}$ below (resp. above) $H$. Let $V_+$ (resp. $V_-$) be the union of some nonempty collection of compressionbodies such that each component of $V_+$ (resp. $V_-$) shares a negative boundary component with $U_-$ (resp. $U_+$). Let $J_\pm$ be the result of amalgamating $H$ with each component of $\boundary_+ V_\pm$. Let $\mc{J}$ be the either the result obtained from $\mc{H}$ by replacing $H \cup \boundary_+ V_+ \cup (\boundary_- V_+ \cap \boundary_- U_-)$ with $J_+$ or the result obtained from $\mc{H}$ by replacing $H \cup \boundary_+ V_- \cup (\boundary_- V_- \cap \boundary_- U_+)$ with $J_-$. In either case, we say that $\mc{J}$ is obtained by an \defn{amalgamation} of $\mc{H}$.
\end{definition}

\begin{lemma}
If $\mc{J}$ is obtained from an oriented generalized Heegaard surface $\mc{H}$ by an amalgamation, then $\mc{J}$ is an oriented generalized Heegaard surface. 
\end{lemma}
\begin{proof}
Use the notation of Definition \ref{amalg def2}. Without loss of generality, consider the case when we amalgamate $H$ with $\boundary_+ V_+$. We may also consider just the case when $V_+$ consists of a single compressionbody; if it consists of more, just iterate the following argument.

Let $F = \boundary_- V_+ \cap \boundary_- U_-$. Let $S = \mc{H}^- \setminus F$. One component $M_0 \subset M \setminus S$ contains $F \cup \boundary_+ V_+ \cup H$ and $F$ separates $M_0$ into two submanifolds. One contains $H$ as a Heegaard surface and the other contains $\boundary_+ V_+$ as a Heegaard surface. By Lemma \ref{amalglem}, we may amalgamate $H$ and $\boundary_+ V_+$ across $F$ to obtain $J_+$, which is a Heegaard surface for $M_0$. The surface $J_+$ inherits an orientation from $\boundary_+ V_+$ that is consistent with the orientation on $H$. Regluing $M_0$ to the other components of $M \setminus S$, we see that $\mc{J}$ is an oriented generalized Heegaard surface for $M$.
\end{proof}

Here is the corollary that is key to our proof of Haken's Lemma.
\begin{corollary}\label{amalg intersect}
Suppose that $H_\pm$ are thick surfaces amalgamated across thin surfaces $F$ to obtain a thick surface $H$. Suppose the amalgamation uses a complete set of discs $\Delta_-$ such that after $\boundary$-reducing $U_-$ using $\Delta_-$ each component of $F \times \{1\} \subset U_-$ contains a single scar. Then for each component $F_0 \subset F$, the thick surface $H$ intersects $F_0$ in a single simple closed curve.
\end{corollary}
\begin{proof}
The intersection between $H$ and $F$ occurs on the frontiers of the tubes $A_-$. By our choice of $\Delta_-$, each component of $F$ intersects a single tube of $A_-$ in a single closed curve.
\end{proof}

We now turn to the relationship between acyclicity and amalgamation.

\begin{definition}
A \defn{height function} on an oriented generalized Heegaard surface $\mc{H}$ is a function $f \co \mc{H} \to \N$ which is constant on each component of $\mc{H}$ and satisfies:
\begin{itemize}
\item for every component $S \subset \mc{H}$, $f(S)$ is odd if and only if $S \subset \mc{H}^+$;
\item if $C$ is a compressionbody of $M\setminus \mc{H}$ with $\boundary_- C \cap \mc{H}^- \neq \nil$ such that the transverse orientation on $\boundary_+ C$ points out of (resp. into) $C$, then, if $F$ is a component of $\boundary_- C$ we have $f(F) < f(\boundary_+ C)$ (resp. $f(F) > f(\boundary_+ C)$).
\end{itemize}
\end{definition}

\begin{lemma}\label{height funct lem}
An oriented generalized Heegaard surface admits a height function if and only if it is acyclic.
\end{lemma}
\begin{proof}
Suppose that $\mc{H}$ is an oriented acyclic generalized Heegaard surface. Each edge of the dual digraph $\Gamma$ corresponds to a component of $\mc{H}$. For an edge $S \subset \mc{H}$, let $f(S)$ equal the \emph{most} number of edges in a coherent path from a source of $\Gamma$ to the head of the edge $S$ (including $S$ itself in the count.) Since $\Gamma$ is acyclic and has only finitely many edges, this is well-defined and satisfies the definition of height function.

Conversely, if the dual digraph to $\mc{H}$ admits a coherent cycle, it cannot also admit a height function as we would contradict the total ordering of $\N$.
\end{proof}

\begin{definition}\label{determined}
Suppose that $\mc{H}$ is an oriented acyclic generalized Heegaard surface with a height function $f$. Suppose that $C$ is a compressionbody of $M \setminus \mc{H}$ and that $H = \boundary_+ C$. Let \[\mc{C}(C) = \{H_i :  \text{that there exists a compressionbody $C'_i \subset M \setminus \mc{H}$ with $H_i = \boundary_+ C'_i$ and $\boundary_- C'_i \cap \boundary_- C \neq \nil$.}\}\] If $C$ is below $H$ (resp. above), let $\mc{A}(C)$ be the subset of $\mc{C}(C)$ on which $f$ is maximal (resp. minimal), where the optimization is over all thick surfaces in $\mc{C}(C)$.  Suppose that $\mc{J}$ is obtained by an amalgamation of $\mc{H}$ and that there exists a compressionbody $C \subset M \setminus \mc{H}$ so that $\boundary_+ C$ is amalgamated with all thick surfaces in $\mc{A}(C)$, then we say that $\mc{J}$ is obtained by an amalgamation of $\mc{H}$ that is \defn{consistent} with the height function $f$. 
\end{definition}

\begin{lemma}
If $\mc{J}$ is obtained from $\mc{H}$ by an amalgamation consistent with a height function $f$, then $\mc{J}$ is acyclic and $f$ naturally induces a height function for $\mc{J}$.
\end{lemma}
\begin{proof}
Use the notation of Definition \ref{determined}. Without loss of generality, assume that $C$ is below $H$. This implies that each $C'_i$ is above $H_i$. Let $C_i$ be the compressionbody below $H_i$ and $C'$ the compressionbody above $H$. Let $H'$ be the surface resulting from the amalgamation of $H$ and $H_i$ across components of $\boundary_- C$ . Let $U$ and $V$ be the compressionbodies above and below $H'$ respectively. Let $\mu = f(H_i)$ and note this is the same for all $i$.

Define $f' \co \mc{J} \to \N$ to be the function which is constant on components of $\mc{J}$ and is defined as follows for a component $S \subset \mc{J}$:
\begin{itemize}
\item If $S = H'$, then $f'(S) = \mu$
\item If $S \subset \boundary_- V \cap \boundary_- C$, then $f'(S) = \mu - 1$
\item $f'(S) = f(S)$, otherwise.
\end{itemize}
(If $C$ were above $H$, we would adjust the above by requiring $f'(S) = \mu + 1$ for any $S \subset \boundary_- U \cap \boundary_- C \cap \mc{J}$.) See Figure \ref{fig: heightamalg} for an example.
Observe that $f'(S)$ is odd if and only if $S \subset \mc{J}^+$.

Notice that $\boundary_- U$ is the union of $\boundary_- C'$ with $\bigcup\limits_i \boundary_- C'_i \setminus \boundary_- C$. If $F$ is a component of $\boundary_- C'_i \setminus \boundary_- C$ for some $i$, then
\[
f'(F) = f(F) > f(H_i) = f'(H').
\]
If $F$ is a component of $\boundary_- C'$, we have
\[
f'(F) = f(F) > f(H) > f(H_i) = f'(H').
\]

Similarly $\boundary_- V$ is the union of $\boundary_- C \setminus \bigcup\limits_i \boundary_- C'_i $ with $\bigcup\limits_i\boundary_- C_i$. If $F$ is a thin surface of $\mc{J}$ lying in $\boundary_- C \setminus \bigcup\limits_i \boundary_- C'_i$, let $C'' \subset M\setminus \mc{H}$ be the compressionbody on the opposite side of $F$ from $C$. By definition of $\mc{A}(C)$,
\[
f'(F) = \mu - 1 > \mu - 2 \geq f(H'') = f'(H'').
\]
Finally, if $F$ is a thin surface of $\mc{J}$ lying in $\boundary_- C_i$ for some $i$, then
\[
f'(F) = f(F) < f(H_i) = f'(H').
\]
Thus, $f'$ is a height function on $\mc{J}$. By Lemma \ref{height funct lem}, $\mc{J}$ is acyclic.
\end{proof}

\begin{figure}[ht!]
\labellist
\small\hair 2pt
\pinlabel {$1$} [b] at 133 54
\pinlabel {$2$} [b] at 49 77
\pinlabel {$3$} [b] at 49 101
\pinlabel {$4$} [b] at 49 124
\pinlabel {$5$} [b] at 49 147
\pinlabel {$6$} [b] at 49 171
\pinlabel {$15$} [b] at 49 195
\pinlabel {$18$} [b] at 49 219
\pinlabel {$26$} [b] at 133 220
\pinlabel {$15$} [b] at 133 195
\pinlabel {$6$} [b] at 133 170
\pinlabel {$5$} [b] at 133 147
\pinlabel {$4$} [b] at 133 125
\pinlabel {$101$} [b] at 133 256
\pinlabel {$98$} [b] at 228 219
\pinlabel {$13$} [b] at 228 194
\pinlabel {$10$} [b] at 228 170

\pinlabel {$1$} [b] at 468 52
\pinlabel {$2$} [b] at 383 75
\pinlabel {$3$} [b] at 383 98
\pinlabel {$4$} [b] at 383 123
\pinlabel {$5$} [b] at 383 145
\pinlabel {$6$} [b] at 383 168
\pinlabel {$15$} [b] at 418 245
\pinlabel {$6$} [b] at 468 168
\pinlabel {$5$} [b] at 468 144
\pinlabel {$4$} [b] at 468 121
\pinlabel {$14$} [b] at 563 215
\pinlabel {$13$} [b] at 563 192
\pinlabel {$10$} [b] at 563 168

\endlabellist
\includegraphics[scale=0.5]{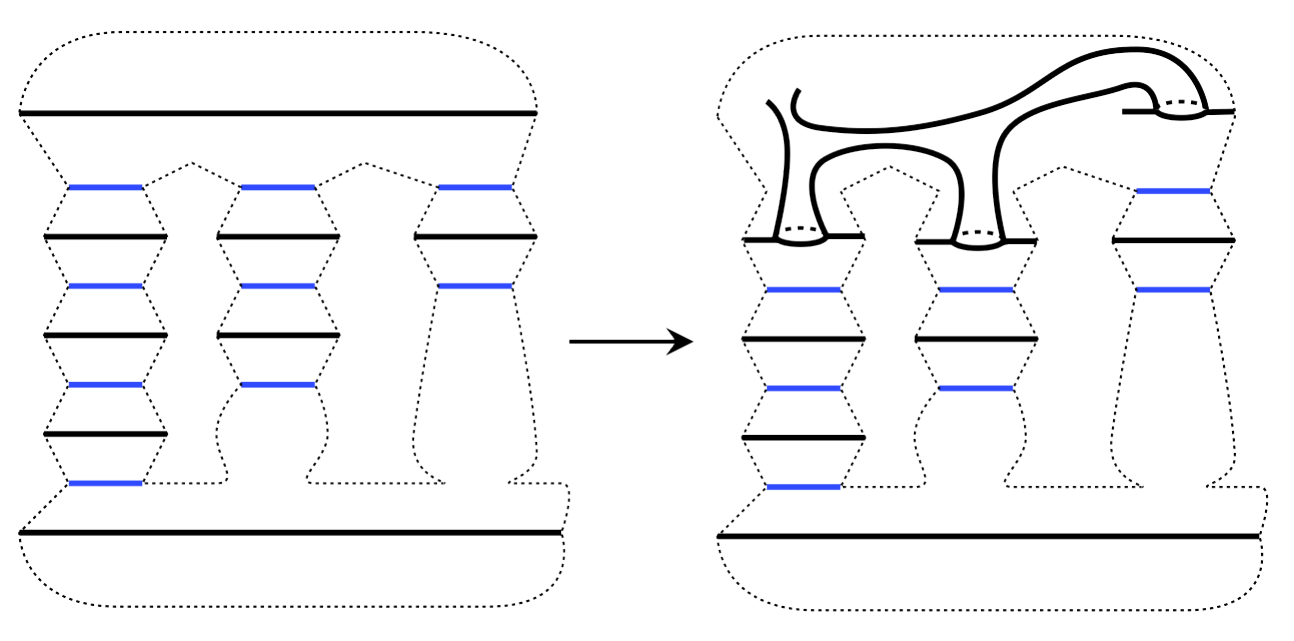}
\caption{An example of an amalgamation respecting a height function. The numbers indicate the heights of the thick and thin surfaces. The top thick surfaces are amalgamated with the next two thick surfaces in the left two columns but not the next thick surface in the rightmost column. Those are the two thick surfaces that are amalgamated with the topmost thick surface since they have the largest heights of the thick surfaces immediately adjacent to the topmost thick surface. The fact that the top right thin surface has a higher height is irrelevant.}
\label{fig: heightamalg}
\end{figure}

\begin{corollary}[cf. {Lackenby \cite[Prop. 3.1]{Lackenby}}]\label{Lack thm}
Suppose that $\mc{H}$ is an oriented generalized Heegaard surface with height function $f$. Then there is an oriented Heegaard surface $H$ obtained from $\mc{H}$ by a sequence of amalgamations consistent with the induced height functions. Furthermore, up to isotopy, $H$ is independent of both $f$ and the sequence of amalgamations.
\end{corollary}
\begin{proof}
Let $\mc{H}$ be an oriented generalized Heegaard surface with height function $f$ and let $s$ be a sequence of amalgamations consistent with the induced height functions. For any compressionbody $C \subset M \setminus \mc{H}$ with $C$ below $\boundary_+ C$, if $F$ is a component of $\boundary_- C$ with $f(F) < f(\boundary_+ C)$, add $(f(\boundary_+ C) - f(F) + 1)/2$ parallel copies of $F$ to $\mc{H}$. One component in each copy is a thick surface and one is a thin surface and, together with $F$ and $\boundary_+ C$, the thick and thin surfaces alternate. Observe that we can restore $\mc{H}$ by consolidating the parallel copies we just inserted.  We can extend both the height function $f$ and the sequence of amalgamations to incorporate these new copies. As amalgamating across a product region is the same as consolidation (Lemma \ref{amalg prod}), this does not affect the result. Consequently, the result of amalgamation consistent with a height function, is independent of the height function, depending only on $\mc{H}$ and on the order of thick surfaces chosen in $s$. We now show how to eliminate even this latter dependency.

Let $\mc{H}'$ be the new generalized Heegaard surface after one of the amalgamations. Its dual digraph is the result of merging edges of the dual digraph of $\mc{H}$. Observe that if $C$ is a compressionbody below $\boundary_+ (C)$, then 
\[
\boundary_- C \cap \mc{H}' \subset f^{-1}(f(\boundary_+ C) - 1)
\]
and if $C$ is above $\boundary_+ C$ then
\[
\boundary_- C \cap \mc{H}' \subset f^{-1}(f(\boundary_+ C) + 1).
\]
We may as well assume, therefore, that for each term of the amalgamation sequence $s$, there is an $i \in \N$ such that the thick surfaces at height $i$ are amalgamated with the thick surfaces at height $i - 1$. It then follows from \cite[Prop. 3.1]{Lackenby} that $H$ is, up to isotopy, independent of $s$ and of the choice of disc sets used in each amalgamation.
\end{proof}

We can now prove additivity of Heegaard genus and  Haken's Lemma. Although the former follows from the latter, it is instructive to prove it independently.

\begin{theorem}[Additivity of Heegaard genus]
If $M$ contains a nonseparating sphere $S$, then $g(M|_S) = g(M) - 1$ and if a connected 3-manifold $M$ is the connected sum of $M_1$ and $M_2$, then $g(M) = g(M_1) + g(M_2)$.
\end{theorem}
\begin{proof} 
The case when $\boundary M$ contains spheres follows relatively easily from the case when it does not. (Capping off spherical components of $\boundary M$ with 3-balls does not affect the Heegaard genus of $M$. It can, however convert an essential sphere $S$ into an inessential sphere. However, if that occurs, the sphere $S$ bounds a punctured 3-ball in $M$, which has Heegaard genus 0.) For simplicity, assume that $\boundary M$ does not contain spheres.  For a generalized Heegaard surface $\mc{H}$, let $\netchi(\mc{H}) = -\chi(\mc{H}^+) + \chi(\mc{H}^-)$. It is easy to verify that if $\mc{H}$ thins to $\mc{J}$ or if $\mc{J}$ amalgamates to $\mc{H}$, then $\netchi(\mc{H}) = \netchi(\mc{J})$. We treat the case of separating and nonseparating spheres separately; in each case, proving two inequalities.

We start by proving the statement for nonseparating spheres.

Suppose that $M$ contains a nonseparating sphere $P$. Let $H$ be a minimal genus Heegaard surface for $M|_P$. The scars from $P$ are two 3-balls $B_1, B_2$ in $M|_P$. We may isotope $H$ so that those 3-balls lie on the same side of $H$.  On that side of $H$, choose a 3-ball $B$ containing those scars.  Remove the interiors of  $B_1, B_2$ creating two spherical boundary components $F_1, F_2$ of $M\setminus P$. Let $W = S^2 \times I$ and let $H' \subset W$ be a sphere bounding a 3-ball in $W$. Let $M'$ be the result of gluing the components of $\boundary W$ to $F_1 \cup F_2$. Then it is possible to assign orientations to $\mc{H} = H \cup F_1 \cup F_2 \cup H'$ so that it is an oriented, acyclic generalized Heegaard surface with $H, H' \subset \mc{H}^+$ and $F_1, F_2 \subset \mc{H}^-$. Notice that $M'$ is homeomorphic to $M$ and that
\[
\netchi(\mc{H}) = -\chi(H) - 2 + 4 = -\chi(H) + 2.
\]
Amalgamate $\mc{H}$ to a Heegaard surface $J$ for $M'$. It has genus equal to $g(H) + 1$. Thus, $g(M)-1 \leq g(M|_P)$. 

Now suppose that $H$ is a Heegaard surface for $M$ of minimal genus. Thin $H$ to a locally thin acyclic oriented generalized Heegaard surface $\mc{H}$. By Corollary \ref{thin spheres exist}, there exists a set of spheres $\mc{S} \subset \mc{H}^-$ such that $M|_\mc{S}$ is irreducible and so that $\mc{S}$ contains a nonseparating sphere $P_0$. Let $\mc{J} = \mc{H}\setminus P_0$ and observe it is an oriented acyclic generalized Heegaard surface for $M|_{P_0}$ with $\netchi(\mc{J}) = \netchi(\mc{H}) - 2 = 2g(H) - 4$. Consistently with a height function, amalgamate $\mc{J}$ to a Heegaard surface $J$ for $M|_{P_0}$ with $\netchi(J) = \netchi(\mc{J})$. Consequently, $g(J) = g(H) - 1$. Thus, $g(M|_{P_0}) \leq g(M) - 1$. Since, $M|_{P_0}$ is homeomorphic to $M|_P$, we have $g(M|_P) = g(M) - 1$.

We now prove the additivity result. Assume that $M = M_1 \# M_2$. By the previous result and uniqueness of prime decompositions of 3-manifolds, we may assume that none of $M$, $M_1$, $M_2$ contain a nonseparating sphere.

Let $H_i$ be a minimal genus Heegaard surface for $M_i$ (with $i = 1,2$). In $M$, there is a sphere $F$ separating $H_1$ from $H_2$ and with $M|_F = M_1 \cup M_2$. Give each of $H_1, H_2, F$ an orientation so that $H_1 \cup H_2 \cup F$ is a generalized Heegaard surface $\mc{H}$ with $H_1, H_2 \subset \mc{H}^+$ and $F \subset \mc{H}^-$. Amalgamating $H_1$ and $H_2$ across $F$ produces a Heegaard surface $H$ for $M$ with \[2g(H) -2= \netchi(H) = \netchi(\mc{H})= 2g(H_1) + 2g(H_2) - 2.\] Thus, $g(M) \leq g(M_1) + g(M_2)$.

Now suppose that $H$ is a Heegaard surface for $M$ of minimal genus. Thin $H$ to a locally thin acyclic oriented $\mc{H}$. There exists a set of spheres $\mc{S} \subset \mc{H}^-$ such that $M|_{\mc{S}}$ is irreducible. By uniqueness of prime decompositions of 3-manifolds, we may group the components of $M|_{\mc{S}}$ into those that are prime factors of $M_1$ and those that are prime factors of $M_2$. As above, by reconstituting $M_1$ and $M_2$ from their factors, we may assemble the components of $\mc{H}\setminus \mc{S}$ into oriented, acyclic generalized Heegaard surfaces $\mc{H}_1$ and $\mc{H}_2$ for $M_1$ and $M_2$. (This may involve replacing components of $\mc{S}$ with other thin spheres.) Since every sphere is separating, by the calculation at the start, we see that $g(M) \geq g(M_1) + g(M_2)$. Thus, equality holds.
\end{proof}

\begin{theorem}[Haken's Lemma]\label{Haken Lem}
Every Heegaard surface $H$ for a compact, connected, orientable reducible 3-manifold $M$ is reducible.
\end{theorem}

\begin{proof}
As we mentioned in the introduction, if $\boundary M$ contains a sphere, then every $\boundary$-parallel sphere is a reducing sphere. We can easily isotope $H$ to intersect such a sphere in a single closed curve. (Likely the curve is inessential on $H$.) Hence, if $\boundary M$ contains a sphere, $H$ is reducible. On the other hand, we can also show that if $M$ contains an essential sphere (whether or not $\boundary M$ contains spheres), then there is an essential sphere in $M$ intersecting $H$ in a single simple closed curve. When $\boundary M$ contains spheres, this is a marginally stronger version of Haken's lemma than what we have stated.

Assume, therefore, that $M$ contains an essential sphere. In particular, it is not the connected sum of two or three 3-balls. Let $\wihat{M}$ be the result of capping off $\boundary M$ with 3-balls. Any essential sphere in $\wihat{M}$ that is disjoint from the 3-balls we used to cap off $\boundary M$ is also essential in $M$. If $\wihat{M}$ has no essential spheres, then $M$ is the connected sum of at least four 3-balls; in particular $\boundary M$ contains at least four spheres. The Heegaard surface $H$ is also a Heegaard surface for $\wihat{M}$.

Suppose, first, that $M$ is the connected sum of at least four 3-balls. Let $P_1, P_2$ be two spheres of $\boundary M$ lying on the same side of $H$. Let $P \subset M \setminus H$ be a sphere separating $P_1$ and $P_2$ from both $H$ and all the other spheres in $\boundary M$. Then $P$ is an essential sphere in $M$. We may isotope it so that it intersects $H$ in a single simple closed curve; likely the curve is inessential in $H$. Hence, there is an essential sphere in $M$ intersecting $H$ in a single closed curve, as claimed.

It remains to consider the case when $\wihat{M}$ contains an essential sphere. Let $\mc{H}$ be a locally thin generalized Heegaard surface for $\mc{H}$ obtained by thinning $H$ (Theorem \ref{locally thin}). By Corollary \ref{thin spheres exist}, $\mc{H}^-$ contains at least one essential (in $\wihat{M}$) sphere $P$. By Lemma \ref{amalglem} and Corollary \ref{amalg intersect}, we may amalgamate $\mc{H}$ to a Heegaard surface $J$ such that $|J \cap P| = 1$. By Theorem \ref{Lack thm}, $J$ is isotopic to $H$. That isotopy takes $P$ to a reducing sphere $P'$ for $H$ in $\wihat{M}$. As $H$ is disjoint from the 3-balls $\wihat{M} \setminus M$, we may further isotope $P'$, relative to $H$, to be disjoint from those 3-balls. The sphere $P'$ is then an essential sphere in $M$ intersecting $H$ in a single simple closed curve, as desired.
\end{proof}

We can also prove the Structure Theorem for Weakly Reducible Heegaard surfaces.

\begin{theorem}[Structure Theorem for Weakly Reducible Heegaard Surfaces]\label{GCT}
Suppose that $H$ is an oriented weakly reducible Heegaard surface for a compact, orientable 3-manifold (possibly with boundary) $M$. Then there exists an oriented incompressible surface $F\subset M$ such that:
\begin{enumerate}
\item For each component $F'$ of $F$, $F' \cap H$ consists of a single simple closed curve, which is inessential in $F'$. The curves of $F \cap H$ bound discs in $F$ all on the same side of $H$; we may perform the construction so that the discs lie on which ever side we wish.
\item $F \cap H$ separates $H$.
\item For each component $F'$ of $F$, $-\chi(F') < -\chi(H)$.
\item There is a weak reducing pair of discs separated by $F$. 
\item If some component of $F$ is a 2-sphere bounding a 3-ball, then there is a stabilizing pair of discs for $H$ disjoint from $F$
\item If some component of $F$ is $\boundary$-parallel, then there is a $\boundary$-stabilizing disc for $H$ disjoint from $F$.
\item There is a generalized Heegaard surface $\mc{H}$ for $M$ such that $\mc{H}^+$ is isotopic to $H+F$, the Haken sum of $H$ and $F$, and $\mc{H}^- = F$; furthermore, the thick surfaces $\mc{H}^+$ are strongly irreducible.
\end{enumerate}
\end{theorem}
\begin{proof}
The proof is much the same as our proof of Haken's Lemma. Let $\mc{H}$ be a locally thin generalized Heegaard surface obtained by thinning a weakly reducible Heegaard surface $H$ (Theorem \ref{locally thin}). Let $F = \mc{H}^-$ and note that $\mc{H}^+$ is strongly irreducible. As before, we may amalgamate $\mc{H}$ to a Heegaard surface $J$ such that for each component $F' \subset \mc{H}^-$, we have $|F' \cap J| = 1$. By the construction, the curve $F' \cap J$ is essential in $J$ and inessential in $F'$. We note that in our process of amalgamation, we always extend the feet of 1-handles lying on the thick surface above a thin surface through the thin surface to lie on the thick surface below the thin surface. Consequently, each curve in $F \cap H$ bounds a disc in $F$ below $H$. We could instead have extended in the other direction, ensuring that each curve of $F \cap H$ bounds a disc in $F$ lying above $H$. This would be equivalent to reversing the orientation on $\mc{H}$, performing the amalgamation as we have described it, and then reverting the orientation back to the original. In any case, it is straightforward to confirm that the Haken sum $H + F$ of $H$ and $F$ is, after a small isotopy, disjoint from $F$ and is, in the complement of $F$, isotopic to $\mc{H}^+$ as claimed. This shows Conclusions (1) and (7).

In $M$, the dual digraph to $F$ can be constructed from the dual digraph to $\mc{H}$ by collapsing each edge corresponding to a thick surface to a single vertex. Thus, the dual digraph to $F$ is acyclic. Consequently, $F$ separates $M$ and, therefore, separates $J$. This is Conclusion (2). 

If $F' \subset F$ is a component, observe that it can be obtained from $H$ by a sequence of compressions, at least one of which is a compression along a compressing (rather than semi-compressing) disc. Thus, $-\chi(F') < -\chi(H)$. This is Conclusion (3).

Consider the construction of $J$. When we amalgamate, according to our construction, we extend a foot of a single 1-handle through each component of $F$. Let $D_-$ be the core disc for one of those 1-handles. In our construction, $D_-$ is an s-disc for a thick surface $H_+$ such that $D_-$ is in the compressionbody below $H_+$. The foot is extended so as to lie on a Heegaard surface $H_-$ in such a way that it misses the scars from a complete set of s-discs for the compressionbody above $H_+$. Let $D_+$ be one of those s-discs. Then $F$ separates $D_-$ from $D_+$. In subsequent amalgamations, although we do need to slide the feet of 1-handles onto other one handles to ensure that when we amalgamate each component of $F$ intersects $J$ in a single loop, we can do so without sliding the foot of a 1-handle across $F$; thus, discs $D_-$ and $D_+$ remain disjoint and separated by $F$ throughout the amalgamation process. This is Conclusion (4).

Suppose that some component $F'$ of $F$ is a 2-sphere bounding a 3-ball in $M$. Since all closed surfaces in 3-balls are either spheres or are compressible, we may pass to an innermost such sphere. It bounds a 3-ball containing a unique component $H'$ of $\mc{H}^+$. That component is not a 2-sphere, for then $H'$ and $F$ could be amalgamated, contradicting the thinness of $\mc{H}$. Since $H'$ is strongly irreducible, by Waldhausen's classification of Heegaard surfaces of $S^3$, it is a torus and admits a stabilizing pair. When performing the amalgamations, we can ensure that the stabilizing pair remains disjoint from $F$. This is Conclusion (5).

If some component $F'$ of $F$ is $\boundary$-parallel, a similar argument shows the existence of a destabilizing disc disjoint from $F$. We use Scharlemann-Thompson's classification of Heegaard splittings of $\text{(surface)} \times I$ in place of the Waldhausen classification. This is Conlcusion (6).
\end{proof}

In the next section, we introduce the operation of ``juggling'' a generalized Heegaard surface. It is an operation that allows us to move spherical thin surfaces from one side of a thick surface to the other side and is used to eliminate inessential curves of intersection between a generalized Heegaard surface and an essential surface. Unfortunately, it can also have the effect of introducing coherent cycles into the dual digraph. Such cycles will always have an edge corresponding to a nonseparating sphere. In order to accommodate those, before proceeding, we generalize our previous work to allow for ``spherical self-amalgamation.'' See Figure \ref{fig: selfamalg} for a depiction.

\begin{figure}[ht!]
\labellist
\small\hair 2pt
\pinlabel {$A$} at 172 406
\pinlabel {$B$} at 172 347
\pinlabel {$H$} [l] at 359 376
\pinlabel {$P$} [b] at 260 360
\pinlabel {$P$} [t] at 85 389
\pinlabel {$H'$} [l] at 359 58
\pinlabel {$H|_{\Delta_A}$} [l] at 359 199
\endlabellist
\centering
\includegraphics[scale=0.5]{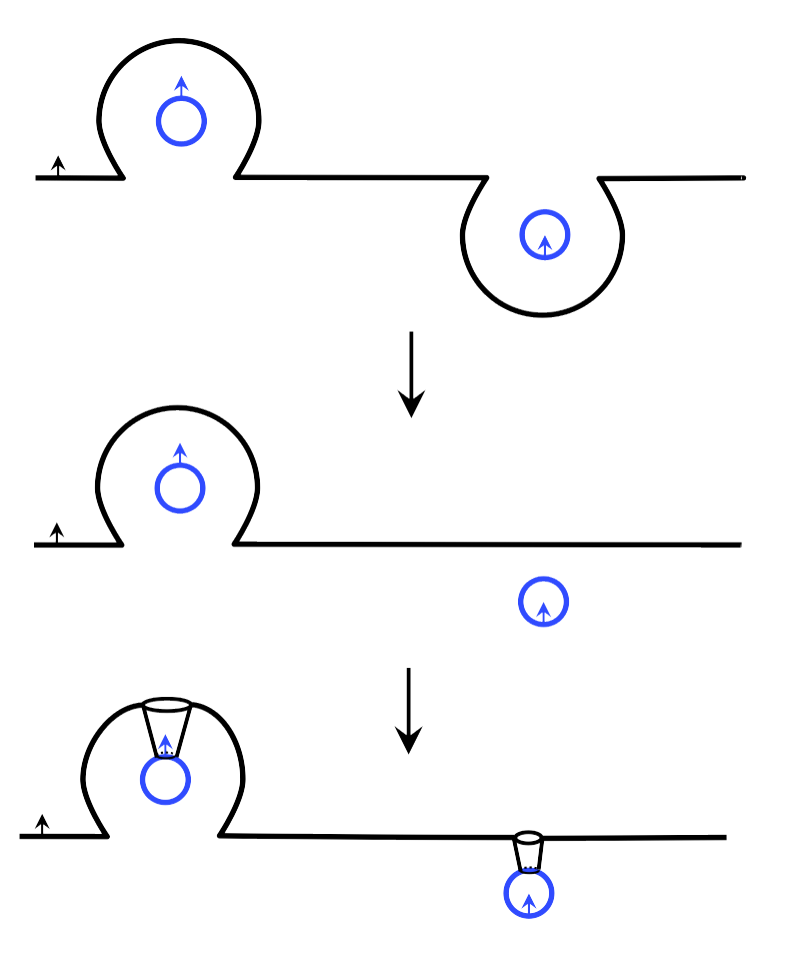}
\caption{An example of spherical self-amalgamation of a Heegaard surface $H$ across a nonseparating sphere $P$. The compressionbody above $H$ in $M\setminus P$ is $A$ and the one below $H$ is $B$. The figure shows the two scars of $P$ in $M|_P$, one on either side of $H$. The Heegaard surface $H'$ is formed by moving $H$ across one of the scars (by compressing along a semi-compressing disc) and then attaching a tube passing through $P$. }
\label{fig: selfamalg}
\end{figure}

\begin{definition}
Suppose that $M$ is a 3-manifold containing a nonseparating collection of mutually disjoint spheres $P$ and that $H$ is an oriented Heegaard surface for $M \setminus P$ such that for each component $P_0 \subset P$ the surface $H$ separates the two scars of $P_0$. Let $H' \subset M$ be the surface constructed as follows. Let $A$ and $B$ be the compressionbodies of $M\setminus P$ above and below $H$ respectively. Let $\Delta_A \subset A$ be a collection of semi-compressing discs for $H$ such that $A|_{\Delta_A}$ is the disjoint union of $P \times I$ and a compressionbody equivalent to the result of capping off the components of $\boundary_- A \cap P$ with 3-balls. Assume that $\Delta_A$ was chosen so that there is exactly one scar in each component of $P \times I$. Let $\alpha$ be a collection of arcs with endpoints in $H|_{\Delta_A}$ and interiors disjoint from $H|_{\Delta_A}$ chosen so that:
\begin{enumerate}
\item Each arc in $\alpha$ has exactly one endpoint in a scar of $H|_{\Delta_A}$. 
\item Each scar from $\Delta_A$ contains exactly one endpoint of $\alpha$
\item Each component of $\alpha$ is the union of an arc that is a cocore of the 2-handle defined by a disc in $\Delta_A$, a vertical arc in one of the components of $P \times I$, and a vertical arc in $A$.
\end{enumerate}
Let $H'$ be the result of attaching tubes to $H$ along the arcs of $\alpha$. Note that $H'$ is orientable and inherits a transverse orientation from $H$. We say that it is the result of a \defn{spherical self-amalgamation of $H$ across $P$}.
\end{definition}

\begin{lemma}\label{lem: self-amalg}
If $H'$ is obtained from $H$ by a spherical self-amalgamation across $P$, then $H'$ is a Heegaard surface for $M$ isotopic to the connected sum of the Heegaard surface $H$ in $M|_P$ with $|P|$ copies of the genus 1 Heegaard surface of $S^1 \times S^2$.
\end{lemma}
\begin{proof}
As above, let $A$ and $B$ be the compressionbodies above and below $H$ in $M\setminus P$, respectively. Let $P \times [-1,1]$ be a regular neighborhood of $P$ in $M$ so that $P \times [-1,0] \subset B$ and $P \times [0, 1] \subset A$.  The $\boundary$-reduction of $A$ along $\Delta_A$ results in compressionbodies embedded in $M$. The compressionbody equivalent to the result of capping of the components of $\boundary_- A \cap P$ does not contain any sphere of $P$. Tubing $H$ along the arcs of $\alpha$ adds 1-handles to the positive boundary of that compressionbody and so it is a compressionbody $A'$ embedded in $M$. Hence, $H'$ is separating. The other side $B'$ of $H'$ can be constructed as follows. Begin by letting $B' = B$. Let $\Delta_B$ be a complete set of compressing and semi-compressing discs for $B$, each disjoint from $\alpha$. The result of $\boundary$-reducing $B$ along $\Delta_B$ is the union of 3-balls and product compressionbodies whose union contains $P \times [-1,0]$. When we $\boundary$-compress $A$ along $\Delta_A$, include $P \times [0,1]$ into $B'$. Thus, $P \subset B'$ and $B' \setminus P$ is equivalent to $B$. The result of $\boundary$-reducing $B'$ along $\Delta_B$ is now the union of 3-balls, product compressionbodies, and components homeomorphic to $(P_0 \times [-1,1])$, one for each component $P_0$ of $P$. Removing a regular neighborhood of $\alpha$ from $B'$ to make it part of $A'$, converts each of those into a $D^2 \times [-1,1]$, which is a 3-ball. Undoing the $\boundary$-reductions along $\Delta_B$, we see that $B'$ is a compressionbody. 

In $H$, for each arc in $\alpha$ choose an arc in $H$ joining its endpoints. Choose the collection $\beta$ of these arcs to be pairwise disjoint. Note that the regular neighborhood of $P \cup \beta$ consists of $|P|$ copies of $S^2 \times S^1$ with an open 3-ball removed from each. The frontier of a regular neighborhood of $\beta$ witnesses the fact that $H'$ is the connected sum of $H$ with $|P|$ copies of the genus 1 Heegaard surface for $S^2 \times S^1$.
\end{proof}

\begin{definition}\label{def: amalg-ob}
Suppose that $\mc{H}$ is an oriented generalized Heegaard surface such that each coherent cycle in the dual digraph contains an edge corresponding to a thin sphere. If $P \subset \mc{H}^-$ is the union of spheres, we say that it is a \defn{minimal complete nonseparating collection of spheres} if $M \setminus P$ is connected and $\mc{H} \setminus P$ in $M\setminus P$ is acyclic.  We say that a generalized Heegaard surface $\mc{J}$ for $M$ is \defn{amalgamation-obtained} from $\mc{H}$ if there is a minimal complete nonseparating collection of spheres $P \subset \mc{H}^-$ and a height function $f$ on $\mc{H} \setminus P$ for $M\setminus P$ such that $\mc{J}$ is obtained by a sequence of amalgamations of $\mc{H}\setminus P$ consistent with $f$ followed by spherical self-amalgamation along $P$.
\end{definition}

\begin{remark}
Note that each coherent cycle in the dual digraph to $\mc{H}$ contains an edge corresponding to a thin sphere if and only if $\mc{H}$ contains a minimal complete nonseparating collection of spheres.
\end{remark}

\begin{proposition}
Suppose that $\mc{H}$ is an oriented generalized Heegaard surface containing a minimal complete collection of nonseparating spheres. Then there is a Heegaard surface $H$ for $M$ amalgamation-obtained from $\mc{H}$. Furthermore, $H$ is unique up to isotopy.
\end{proposition}
\begin{proof}
The surface $\mc{H}\setminus P$ is an oriented acyclic generalized Heegaard surface for $M\setminus P$. By Lemma \ref{height funct lem}, it admits a height function. By Corollary \ref{Lack thm}, $\mc{H}\setminus P$ may be amalgamated to an oriented Heegaard surface $J$ for $M\setminus P$ and $J$ is independent of the height function and the sequence of amalgamations. Note that $J \cup P$ is an oriented generalized Heegaard surface for $M$. Let $H$ be the result of spherical self-amalgamation of $J \cup P$ across $P$. By Lemma \ref{lem: self-amalg} it is a Heegaard surface for $M$ and it is amalgamation obtained from $\mc{H}$. 

We need only show that $H$ does not depend (up to isotopy) on the choice of $P$. Suppose that $P' \subset \mc{H}^-$ is another minimal complete nonseparating collection of spheres. It is enough to consider the case when $P$ and $P'$ differ by a single sphere. Let $P_0 \subset P$ and $P'_0 \subset P'$ be the sphere such that $P \setminus P_0 = P' \setminus P'_0$. Let $M' = M \setminus (P \cup P'_0) = M \setminus (P' \cup P_0)$. It has two components, and $\mc{H}$ restricts to an acyclic oriented generalized Heegard surface in each. Choose a height function for each and amalgamate consistent with the height function to create Heegaard surfaces $H_1$ and $H_2$. Since $\mc{H}$ was oriented, both $H_1 \cup P_0 \cup H_2$ and $H_1 \cup P'_0 \cup H_2$ are acyclic oriented generalized Heegaard surfaces for $M\setminus P'$ and $M\setminus P$ respectively. Amalgamate these to Heegaard surfaces $J'$ and $J$ respectively. Observe that each of $J'$ and $J$ is isotopic to the result of connect-summing the Heegaard surfaces $H_1$ and $H_2$. As noted above, performing spherical self-amalmation of $J'$ across $P'$ is isotopic to the result of performing the connected sum of $J'$ with $|P'|$ copies of the genus 1 Heegard surface of $S^2 \times S^1$. Similarly, the self-amalgamation of $J$ across $P$ is isotopic to the result of performing the connected sum of $J$ with $|P|$ copies of the genus 1 Heegaard surface of $S^2 \times S^1$. Since $|P| = |P'|$, the resulting Heegaard surfaces are isotopic. 
\end{proof}

\section{Juggling}\label{juggling}

Suppose that $\mc{H}$ is a generalized Heegaard surface and that $Q \subset M$ is a sphere disjoint from $\mc{H}$. We define an operation, we call a ``juggle'' of $Q$. There are three versions; two are depicted in Figure \ref{fig:juggle}.

\begin{figure}[ht!]
\labellist
\small\hair 2pt
\pinlabel {$P$} [r] at 56 594
\pinlabel {$Q$} [bl] at 195 709
\pinlabel {$\alpha$} [l] at 162 553
\pinlabel {$R$} [r] at 14 520
\pinlabel {$P(\alpha)$} [r] at 71 163
\pinlabel {$R(\alpha)$} [r] at 27 83

\pinlabel {$P=R$} [l] at 1040 608
\pinlabel {$Q$} [bl] at 881 727
\pinlabel {$\alpha$} [b] at 884 535
\pinlabel {$P(\alpha) = R(\alpha)$} [l] at 1061 138

\endlabellist
\centering
\includegraphics[scale=0.35]{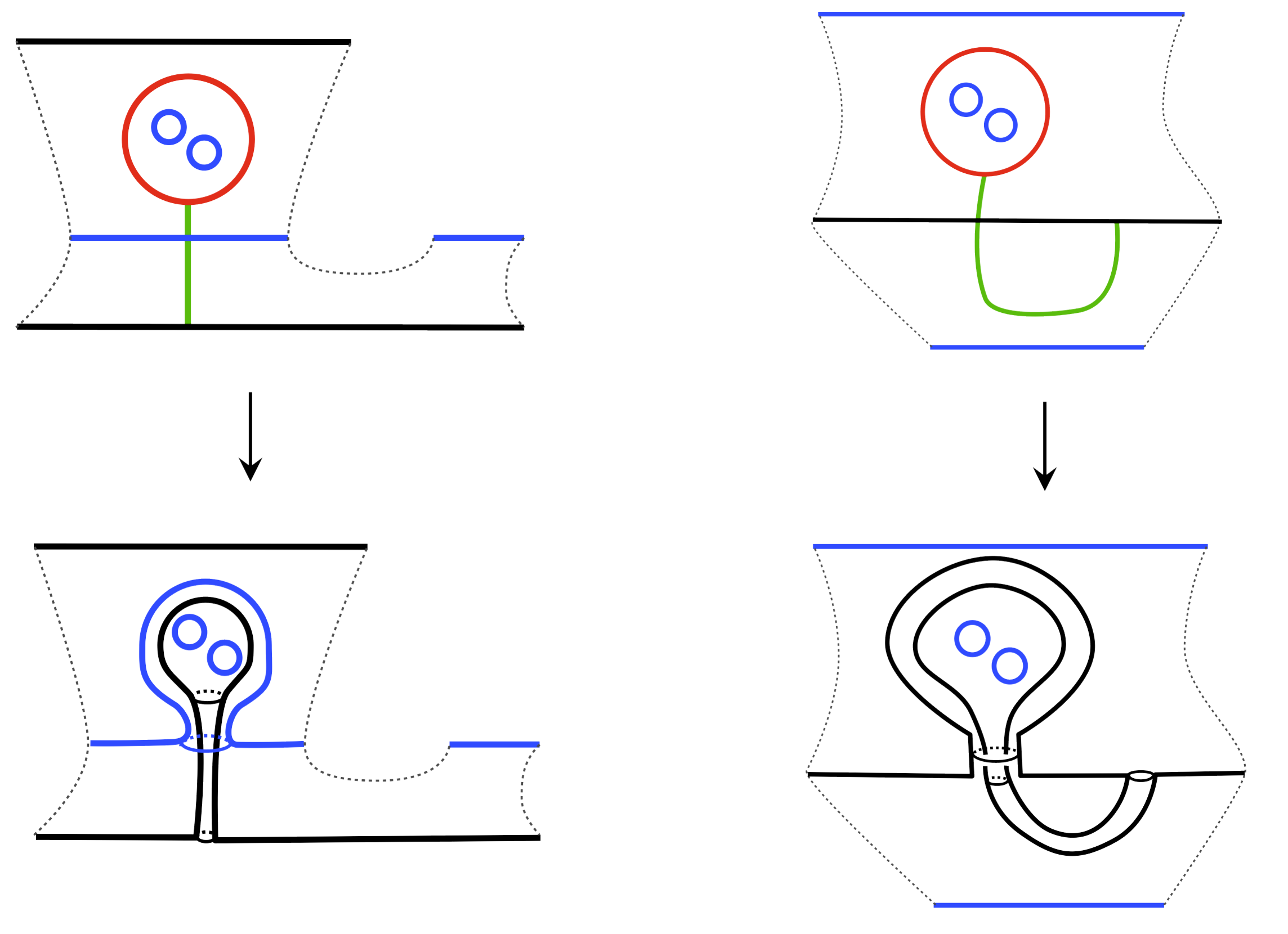}
\caption{The left arrow depicts a juggle (Definition \ref{def: juggle}) where $R$ is a thick surface and $P$ is a thin surface.  The right arrow depicts a juggle where $P = R$ is a thick surface. The path $\alpha$ is shown in green for both juggles. Not shown is the juggle where $P$ is a thick surface and $R$ is a thin surface. The dotted lines simply help delineate the compressionbodies. }
\label{fig:juggle}
\end{figure}

\begin{definition}\label{def: juggle}
Let $Q \subset M\setminus \mc{H}$ be a sphere and let $R$ be a component of $\mc{H}$. Let $C$ be the compressionbody containing $Q$. Suppose that $\alpha$ is an embedded path in the interior of $M$ joining $R$ to $Q$, with interior transverse to $\mc{H}$, disjoint from $Q$, and intersecting $\mc{H}$ exactly once in its interior. Let $P$ be the component of $\mc{H}$ intersecting the interior of $\alpha$. (Possibly $P = R$.) Require that one of $P$ or $R$ is a thick surface. Let $Q_1$ and $Q_2$ be parallel copies of $Q$ in $C$ such that $Q_2$ is between $Q_1$ and $Q$, each intersecting $\alpha$ exactly once (transversally). 

If $P \neq R$, let $P(\alpha)$ be the result of tubing $P$ to $Q_1$ along the portion of $\alpha$ from $P$ to $Q_1$ and let $R(\alpha)$ be the result of tubing $R$ to $Q_2$ along the portion of $\alpha$ from $R$ to $Q_2$, so that $R(\alpha)$ is disjoint from $P(\alpha)$. Give $P(\alpha)$ and $R(\alpha)$ the orientations inherited from $P$ and $R$ respectively. If $P = R$, let $P(\alpha) = R(\alpha)$ be the result of first tubing $P$ along the subarc of $\alpha \setminus P$ joining $P$ to $Q_1$ and then tubing along the entirety of $\alpha$. The surface $\mc{J}$ obtained from $\mc{H}$ by replacing $P$ and $R$ with $P(\alpha)$ and $R(\alpha)$ respectively is obtained from $\mc{H}$ by a \defn{juggle}.
\end{definition}

\begin{remark} By the lightbulb trick, the exact choice of $\alpha$ is immaterial; different choices produce isotopic surfaces.
\end{remark}

\begin{lemma}\label{lem: juggle}
Suppose that $\mc{J}$ is obtained from $\mc{H}$ by a juggle. Then the following hold:
\begin{enumerate}
\item $\mc{J}$ is an oriented generalized Heegaard surface;
\item If $\mc{H}$ is acyclic and $\mc{H}^-$ contains no nonseparating sphere, then $\mc{J}$ is acyclic and $\mc{J}^-$ contains no nonseparating sphere. In which case, any sequence of amalgamations of  $\mc{H}$ and $\mc{J}$ consistent with height functions  create  isotopic (oriented) Heegaard surfaces for $M$; and
\item If $\mc{H}$ is locally thin, then so is $\mc{J}$. 
\end{enumerate}
\end{lemma}
\begin{proof}
Use the notation of Definition \ref{def: juggle}. Let $C$ be the compressionbody containing $Q$. The sphere $Q$ divides $C$ into two compressionbodies. One has $\boundary_+ C$ as its positive boundary and contains $Q$ in its negative boundary. The other $C_Q \subset C$ has $\boundary_+ C_Q = Q$.  Note that $\boundary_- C_Q \subset \boundary_- C$ is the union of spheres. Note also that $P \subset \boundary C$ is not contained in $C_Q$. Observe that if we cut $M$ open along $Q$ and paste in a 3-ball to the corresponding boundary component of $C$, the arc $\alpha$ guides an isotopy of both $P$ and $R$. Thus, each compressionbody of $M\setminus \mc{J}$ is either identical to a compressionbody of $M\setminus \mc{H}$, is obtained from $C$ by attaching a 3-ball to $Q$, or is obtained from a compressionbody of $M\setminus \mc{H}$ by removing a 3-ball from the interior. It follows that $\mc{J}$ is a generalized Heegaard surface. The fact that $\alpha$ intersects $\mc{H}$ exactly once in its interior means that the orientation on $\mc{H}$ induces an orientation on $\mc{J}$. This is Conclusion (1).

To prove (2), assume that every component of $\boundary_- C_Q \cap \mc{H}^-$ is separating. Consider the effect of the juggle on the dual digraph $\Gamma$ to $\mc{H}$. Let $\Gamma'$ be the dual digraph to $\mc{J}$. If $P = R$ or if $\boundary_- C_Q \cap \mc{H}^- = \nil$, then $\Gamma$ and $\Gamma'$ are isomorphic; in particular, $\Gamma'$ is acyclic.  Suppose, therefore, that $P \neq R$ and that $\boundary_- C_Q \cap \mc{H}^- \neq \nil$. One of $P, R$ is a thin surface and the other is a thick surface. Let $A$ be the compressionbody containing $P$ be on the opposite side of $P$ from $C$ and let $B$ be the compressionbody containing $R$ on the opposite side of $R$ from $A$. The effect of the juggle is effectively to move $\boundary_- C_Q$ from the compressionbody $C$ to the compressionbody $B$. Thus, any edge of $\Gamma$ dual to a component of $\boundary_- C_Q \cap \mc{H}^-$ has its end at $C$ slid across the edges dual to $P$ and $R$ to have that end at $B$. By our assumption, each of the edges of $\Gamma$ corresponding to those components separates $\Gamma$. Hence, this slide of endpoints from $C$ to $B$ preserves the fact that the dual digraph is acyclic.

We now explain why amalgamating consistently with the height functions produces isotopic oriented Heegaard surfaces for $M$. Let $F_1, \hdots, F_n$ be the components of $\boundary_- C_Q \cap \mc{H}^-$. For all $i$, the edge of $\Gamma$ corresponding to $F_i$ separates $\Gamma$ into two connected subgraphs.  Let $\Gamma_i$ be the subgraph of $\Gamma \setminus F_i$ not containing the vertex $C$. Since (by our assumption) each of the $F_i$ is a separating sphere, the subgraphs $\Gamma_i$ are all distinct.

Without loss of generality, assume that the orientation on $P$ points into $C$. (If instead it points out of $C$, it is straightforward to adapt the following argument.) We can choose a height function $f$ for $\mc{H}$ such that there are $m_0, m_1, \hdots, m_n$ with $m_0 = -1$ and $m_1, \hdots, m_n$ positive integers such that $f(F_i) = m_i$ and if $S$ is an edge of $\Gamma_i$ then $f(S) \in (m_{i-1}, m_i) \subset \N$. Then $f$ is also a height function on $\mc{J}$. Amalgamate $\mc{H}$ and $\mc{J}$ to Heegaard surfaces in such a way that we respect the height function $f$. Do this by amalgamating the thick surfaces in each $\Gamma_i$ in the same order and consecutively for $i = 1, \hdots, n$. Then amalgamate the thick surfaces not in any $\Gamma_i$ in the same order. Finally, perform the remaining amalgamations in the same order. We see that the resulting Heegaard surfaces are isotopic since, in terms of the induced handle structure on $M$, we are performing the same handle attachments in the same order. This proves Conclusion (2).

Finally, irrespective of whether or not $\mc{H}^-$ contains nonseparating spheres, we prove (3). Suppose that $\mc{H}$ is locally thin. Up to homeomorphism preserving the $\boundary_\pm$ designations, the difference between the compressionbodies of $M\setminus \mc{H}$ and the compressionbodies of $M\setminus \mc{J}$ is either removing spherical negative boundary components by capping them with 3-balls, or introducing spherical negative boundary components by removing open 3-balls from the interior of the compressionbody. These operations do not affect whether or not a compressionbody is a punctured product compressionbody and they do not affect which curves in a thick surface bound compressing discs. Thus, local thinness is preserved.
\end{proof}

We now address the situation when juggling produces coherent cycles in the dual digraph. 

\begin{lemma}\label{don't drop the ball}
Suppose that $\mc{H}$ is an oriented generalized Heegaard surface such that each coherent cycle in its dual digraph contains an edge corresponding to a thin sphere (necessarily nonseparating). Let $\mc{J}$ be obtained from $\mc{H}$ by a juggle. Then $\mc{J}$ is an oriented generalized Heegaard surface such that each coherent cycle in its dual digraph contains an edge corresponding to a thin sphere. Furthermore, if $H$ and $J$ are Heegaard surfaces amalgamation-obtained from $\mc{H}$ and $\mc{J}$ respectively, then $H$ and $J$ are isotopic.
\end{lemma}
\begin{proof}
Use the notation of Definition \ref{def: juggle}. The sphere $Q$ that is used to define the juggle is disjoint from $\mc{H}$. It is contained in a compressionbody $C \subset M\setminus \mc{H}$ and in $C$ bounds a compressionbody $C_Q$ such that $Q = \boundary_+ C_Q$. As before, let $F_1, \hdots, F_n \subset \boundary_- C_Q \cap \mc{H}^-$, each is a sphere. If none are separating, the result follows immediately from Lemma \ref{lem: juggle}, so assume at least one is nonseparating. If the orientation on $P$ points out of $C$, then the edges of the dual digraph $\Gamma$ for $\mc{H}$ corresponding to each $F_i$ point into the vertex $C$. The effect on the juggle on $\Gamma$ is to take the tips of those edges and slide them in the direction of the edges corresponding to $P$ and $R$. Any coherent cycle in $\Gamma$ after such a slide corresponds to a coherent cycle prior to the slide. Consequently, every coherent cycle for the dual digraph to $\mc{J}$ contains an edge corresponding to a thin sphere. This holds even if $P$ or $R$ is a thin sphere, since some $F_i$ is nonseparating. Similar to the proof of  Lemma \ref{lem: juggle}, it is straightforward to see that the Heegaard surfaces amalgamation-obtained from $\mc{H}$ and $\mc{J}$ are isotopic.

Suppose, therefore, that the orientation for $P$ points into $C$. In this case, the edges of the dual digraph $\Gamma$ for $\mc{H}$ corresponding to each $F_i$ point out of the vertex $C$. The effect on the juggle on $\Gamma$ is to take the tips of those edges and slide them \emph{opposite to} the direction of the edges corresponding to $P$ and $R$. This can create cycles in the dual digraph, but each cycle created will contain an edge corresponding to a nonseparating $F_i$. Additionally, any coherent cycle in $\Gamma$ prior to the juggle, survives as a coherent cycle after the juggle, possibly with the addition of the edges corresponding to $P$ and $R$. 

Let $\mc{F}$ be the union of those $F_i$ that belong to a coherent cycle in the dual digraph to $\mc{J}$ but not in $\mc{H}$; however if more than one $F_i$ belong to the same cycle include only one from each cycle. Let $P_H$ be a minimal complete nonseparating collection of spheres in $\mc{H}^-$ and expand that to a minimal complete nonseparating collection of spheres in $\mc{J}^-$ by adding in $\mc{F}$. Choose a height function $f$ on $\mc{H}\setminus P_H$ and let $f'$ be its restriction to $\mc{H}' = \mc{H} \setminus (P_H \cup \mc{F}) = \mc{J} \setminus (P_H \cup \mc{F})$. We may amalgamate $\mc{H}'$ consistent to the height function to a Heegaard surface $L$. The Heegaard surface $H$ is obtained by amalgamating across $\mc{F}$ and spherical-self amalgamation across $P_H$. The Heegaard surface $J$ is obtained by spherical self-amalgamation across $P_H \cup \mc{F}$. However, these are simply two ways of forming the connected sum of $L$ with $|P_H \cup \mc{F}|$ copies of the genus 1 Heegaard surface of $S^2 \times S^1$. Thus, $H$ and $J$ are isotopic.
\end{proof}

\begin{corollary}\label{disjoint spheres}
Suppose that $H$ is a Heegaard surface for $M$ and that $\mc{S}$ is the union of pairwise disjoint essential spheres and discs in $M$. Then there exists an oriented generalized Heegaard surface $\mc{H}$ for $M$ such that:
\begin{enumerate}
\item Every coherent cycle in the dual digraph to $\mc{H}$ contains an edge corresponding to a thin sphere;
\item $H$ is, up to isotopy, amalgamation-obtained from $\mc{H}$;
\item Each sphere component of $\mc{S}$ is disjoint from $\mc{H}$;
\item Each disc component of $\mc{S}$ is disjoint from $\mc{H}^-$ and intersects $\mc{H}^+$ in a single simple closed curve.
\end{enumerate}
\end{corollary}
\begin{proof}
Let $\mc{H}'$ be a locally thin, acyclic, oriented generalized Heegaard surface obtained by thinning $H$. Such exists by Theorem \ref{locally thin}. By Corollary \ref{Lack thm}, $H$ is, up to isotopy, amalgamation-obtained from $\mc{H}'$. Out of all generalized Heegaard surfaces obtained from $\mc{H}'$ by a sequence of juggles and isotopies, let $\mc{H}$ be one which minimizes $|\mc{H} \cap \mc{S}|$. By Lemma \ref{don't drop the ball}, every coherent cycle in the dual digraph to $\mc{H}$ contains a thin sphere and, up to isotopy, $H$ is amalgamation-obtained from $\mc{H}$. We claim that $\mc{H}$ satisfies (3) and (4).

\textbf{Claim 1:}  $\mc{S}$ is disjoint from $\mc{H}^-$ .

Suppose not. Consider an innermost curve $\zeta \subset \mc{S} \cap \mc{H}^-$ that bounds a disc $D \subset \mc{S}$ with interior disjoint from $\mc{H}^-$. Since $\mc{H}^-$ is incompressible, $D$ must be a semi-compressing disc for a component $F \subset \mc{H}^-$. Let $W \subset M\setminus \mc{H}^-$ be the component containing $D$ and let $J = \mc{H}^+ \cap W$. As $\mc{H}$ is locally thin, by Lemma \ref{SISHL}, $J$ can be isotoped so that each component of $J \cap D$ is inessential in $D$ and each innermost disc in $D$ is a semi-compressing disc for $J$. 

\textbf{Case 1a:} $J \cap D = \nil$

In this case, by Lemma \ref{thin s-discs}, the union of $D$ with a disc in $F$ is a sphere $Q$, which we can push slightly into the punctured 3-ball $B$ in $W$ bounded by $F \cup D$ and disjoint from $J$.

Let $A \subset M\setminus \mc{H}$ be the compressionbody on the opposite side of $F$ from $W$. Let $W' \subset M\setminus \mc{H}^-$ be the component containing $A$. In $W'$ we may perform an isotopy of the surface $H' = \mc{H}^+ \cap W'$ so that it is still transverse to $\mc{S}$ and so that there is a path from $H'$ to $Q$ which is disjoint from $\mc{S}$. (This isotopy may increase the number of intersections between $H'$ and $\mc{S}$.) Using this path, perform a juggle of $F$ and $H'$ over $Q$. After this juggle, we have not changed $|\mc{H}^- \cap \mc{S}|$, but the disc $D$ is now no longer a semi-compressing disc. We may therefore isotope $\mc{H}$ to eliminate it without increasing $|\mc{H}^- \cap \mc{S}|$. This contradicts our choice of $\mc{H}$.

\textbf{Case 1b:} $J \cap D \neq \nil$

Consider an innermost disc $E \subset D$ whose boundary is a component of $J \cap D$ that is innermost in $D$. As each component of $J \cap D$ is inessential in $J$, the union of $D$ with a disc $E \subset H$ is a sphere $Q$ bounding a punctured 3-ball in $M$ with interior disjoint from $J\setminus E$. Push $Q$ slightly into this 3-ball. As the interior of $D$ is disjoint from $\mc{H}^-$, there is an embedded path $\alpha$ with one endpoint on either $F$ or $J\setminus E$, one endpoint on $Q$, with interior intersecting $J$ in a single point contained in $E$, and which is otherwise disjoint from $\mc{H} \cup \mc{S}$. Let $R = J$. If $\alpha$ has an endpoint on $F$, let $P = F$; if it has an endpoint on $J$, let $P = J$. Use this path to perform a juggle of $R$ and $P$ over $Q$, as in Definition \ref{def: juggle}. After the juggle, $|\mc{H} \cap \mc{S}|$ is unchanged, but $D$ is no longer a semi-compressing disc. An isotopy of $J$ therefore removes it, decreasing $|J \cap \mc{S}|$. Repeating as necessary, we make $J$ disjoint from $D$. By Case 1a, we therefore contradict our choice of $\mc{H}$.

We conclude that $\mc{S}$ is disjoint from $\mc{H}^-$.  The intersection of $\mc{S}$ with each component of $M \setminus \mc{H}^-$ is therefore the union of components of $\mc{S}$. Let $W \subset M\setminus \mc{H}^-$ be a component and let $J = \mc{H}^+ \cap W$. By Lemma \ref{SISHL}, we may isotope $J$ in $W$ so that for each component $S_0 \subset \mc{S}$, either:
\begin{itemize}
\item $S_0$ is a sphere and $J \cap S_0$ is the union of curves that are inessential in $J$ and the innermost such curves on $S_0$ bound semi-compressing discs for $J$. 
\item $S_0$ is a compressing disc for $\boundary M$ and the compressionbody of $W\setminus J$ containing $\boundary S_0$ is a punctured product compressionbody. The intersection between $S_0$ and $J$ consists of curves that are inessential in $J$, the innermost ones on $S_0$ bound semi-compressing discs for $J$, and exactly one essential simple closed curve on $J$. That curve bounds a compressing disc for $J$ in the compressionbody on the opposite side of $J$ from that containing $\boundary S_0$. 
\end{itemize}

Suppose that some component of $\mc{S} \cap J$ is inessential in $J$. As in Case (1b) above, we may peform juggles, to eliminate those curves of intersection without increasing $|\mc{H} \cap \mc{S}|$. Consequently, we may assume that each component of $\mc{S} \cap W$ is either a sphere disjoint from $J$ or a compressing disc for $\boundary M$ intersecting $J$ in a single simple closed curve, essential on $J$. This again contradicts our choice of $\mc{H}$, and so $\mc{H}$ must satisfy (3) and (4).
\end{proof}

\section{The Strong Haken Theorem}\label{strong haken}

The remainder of the paper is concerned with the proof of the Strong Haken Theorem. The bulk of the work lies in being careful how we amalgamate. For simplicity, assume that $\boundary M$ does not have any spheres; the general case can be relatively easily obtained from this one and is covered in both \cite{Scharlemann} and \cite{HeSc}

\begin{theorem}[Strong Haken]
Suppose that $H$ is a Heegaard surface for $M$, a compact, connected, orientable 3-manifold. Let $\mc{S} \subset M$ be a collection of pairwise disjoint essential spheres and discs. Then $H$ can be isotoped so that for each component $S_0 \subset \mc{S}$, $H \cap S_0$ is a single simple closed curve.  
\end{theorem}

Let $\mathbb{H}$ be the set of oriented generalized Heegaard surfaces for $M$ such that:
\begin{enumerate}
\item Each coherent cycle in $\mc{H}$ contains an edge corresponding to a thin sphere
\item $H$ is, up to isotopy, amalgamation-obtained from $\mc{H}$
\item Each sphere component of $\mc{S}$ is disjoint from $\mc{H}^-$ and is either disjoint from $\mc{H}^+$ or intersects $\mc{H}^+$ in a single simple closed curve that is essential on $\mc{H}^+$.
\item Each disc component of $\mc{S}$ is disjoint from $\mc{H}^-$ and intersects $\mc{H}^+$ in a single simple closed curve that is essential on $\mc{H}^+$.
\item Each compressionbody $C \subset M\setminus \mc{H}$ contains a collection of s-discs $\Delta_C$ such that after $\boundary$-reducing $C$ along $\Delta_C$, each product compressionbody component contains a single scar and each component of $\mc{S} \cap \mc{H}$ is disjoint from that scar. 
\end{enumerate}
By Corollary \ref{disjoint spheres}, there is an $\mc{H}$ satisfying (1) - (4). In fact, Corollary \ref{disjoint spheres} guarantees that we can arrange for each sphere component of $\mc{S}$ to be disjoint from $\mc{H}$ and each disc component to be disjoint from $\mc{H}^-$ and intersect $\mc{H}^+$ exactly once.  In fact, Corollary \ref{disjoint spheres} guarantees that we can arrange for each sphere component of $\mc{S}$  to be \emph{disjoint} from $\mc{H}$. We have made the statement of Condition (3) weaker than this, so that we can preserve the property as we amalgamate. By the Monopod Lemma (Lemma \ref{monopod lemma}),  we may choose the desired $\Delta_C$ so as to satisfy (5).

Without loss of generality, assume that out of all such $\mc{H}$ we choose $\mc{H}$ to have the fewest components. Note that this means that no component of $M \setminus \mc{H}$ is a product compressionbody bounded by a component of $\mc{H}^+$ and a component of $\mc{H}^-$. There may be a product compressionbody incident to a component of $\boundary M$.  Let $P \subset \mc{H}^-$ be a minimal complete nonseparating collection of spheres. Let $f$ be a height function for $M\setminus P$.

Consider the sequence of amalgamations consistent with $f$ and spherical self-amalgamations that produce a Heegaard surface isotopic to $H$. If $\mc{H}\setminus P$ is connected, the sequence consists of only spherical self-amalgamations, otherwise it consists of some number of amalgamations followed by spherical self-amalgamations.

Begin by assuming that $\mc{H}\setminus P$ is disconnected. At each stage of the sequence, the generalized Heegaard surface satisfies (1) and (2). Consider the first step of this sequence and begin by supposing that it is an amalgamation of $J_+$ and $J_-$, each the union of thick surfaces, across $F$, the union of thin surfaces. Let $U_\pm$ be the (union of) compressionbodies on either side of $J_+$, with $U_+$ above $J_+$. Let $V_\pm$ be the (union of) compressionbodies on either side of $J_-$, with $V_+$ above $J_-$. Assume we have picked the notation so that $F = \boundary_- U_- \cap \boundary_- V_+$. If no component of  $\mc{S} \cap U_-$ or $\mc{S} \cap V_+$ separates a component of $F$ from $J_\pm$, then the resulting generalized Heegaard surface still satisfies (1) - (5), and we contradict our choice of $\mc{H}$. 

Consequently, some components of $\mc{S}$ separate $F$ from either $J_-$ or $J_+$. In this case, each such component of $\mc{S}$ is a sphere and the components of $F$ separated from $J_-$ and $J_+$ are also spheres. Suppose that $S_1, \hdots, S_k \subset V_+$ separate components of $F$ from $J_-$. The exterior of $S_1 \cup \cdots \cup S_k$ in $V_+$ is the disjoint union of compressionbodies, where one scar from each $S_i$ is a positive boundary component and one scar from each $S_i$ is a negative boundary component. Let $S'_1, \hdots, S'_{k'} \subset \mc{S} \cap U_-$ be those components separating components of $F$ from $J_+$. As before, their exterior is the union of compressionbodies such that each $S'_i$ produces one scar in positive boundary and one in a negative boundary. Thus when we amalgamate $J_-$ and $J_+$ across $F$, using the discs $\Delta_{U_-}$ and $\Delta_{V_+}$, the resulting thick surface $J$ intersects each $S_i$ and $S'_i$ exactly once in a curve parallel to a component of $\Delta_{U_-}$. If $J$ intersects each $S_i$ and $S'_i$ exactly once, we do not need to perform further modifications. Otherwise, we do the following. For simplicity of exposition assume that each $S_i$ and $S'_i$ intersects $J$ more than once. See Figure \ref{fig: multitube} for a depiction of the following isotopy.

Using the structure of the compressionbodies that are the complement of the union of the $S_i$ in $V_+$ and the $S'_i$ in  $U_-$, we may embed a collection of vertical discs $X$ in $(V_+|_{\Delta_{V_+}}) \cup (U_-|_{\Delta_{U_-}})$ so that each disc is a square with one edge on $J_- $ or $J_+$, one edge on $\mc{S}$ and its other two edges on distinct tubes arising from the vertical extensions of the scars of $\Delta_{U_+}$. Furthermore, these discs may be constructed so that each intersects each $S_i$ and $S'_i$ in at most 1 arc and so that the union of the arcs forms a tree with the components of $S_i \cap J$ and $S'_i \cap J$ as the vertices. These discs guide an isotopy of $J$ so that after the isotopy it intersects each $S_i$ and each $S'_i$ in a single simple closed curve. Since each $S_i$ and $S'_i$ is essential, these simple closed curves are necessarily essential in $J$. As they are obtained by banding together the scars from $\Delta_{U_-}$, an isotopy ensures they are disjoint from the discs $\Delta_{U_-} \cup \Delta_{U_+}$. Afterwards, we have constructed a generalized Heegaard surface satisfying (1) - (5) but with fewer components than $\mc{H}$, a contradiction. 

\begin{figure}[ht!]
\labellist
\small\hair 2pt
\pinlabel {$J_+$} [r] at 1 270
\pinlabel {$J$} [l] at 465 32
\pinlabel {$S'_1$} [l] at 425 213
\pinlabel {$S_1$} [l] at 425 92
\pinlabel {$F$} [l] at 406 153
\pinlabel {$X$} at 180 242
\pinlabel {$X$} at 313 240
\pinlabel {$X$} at 313 57
\endlabellist
\includegraphics[scale=0.5]{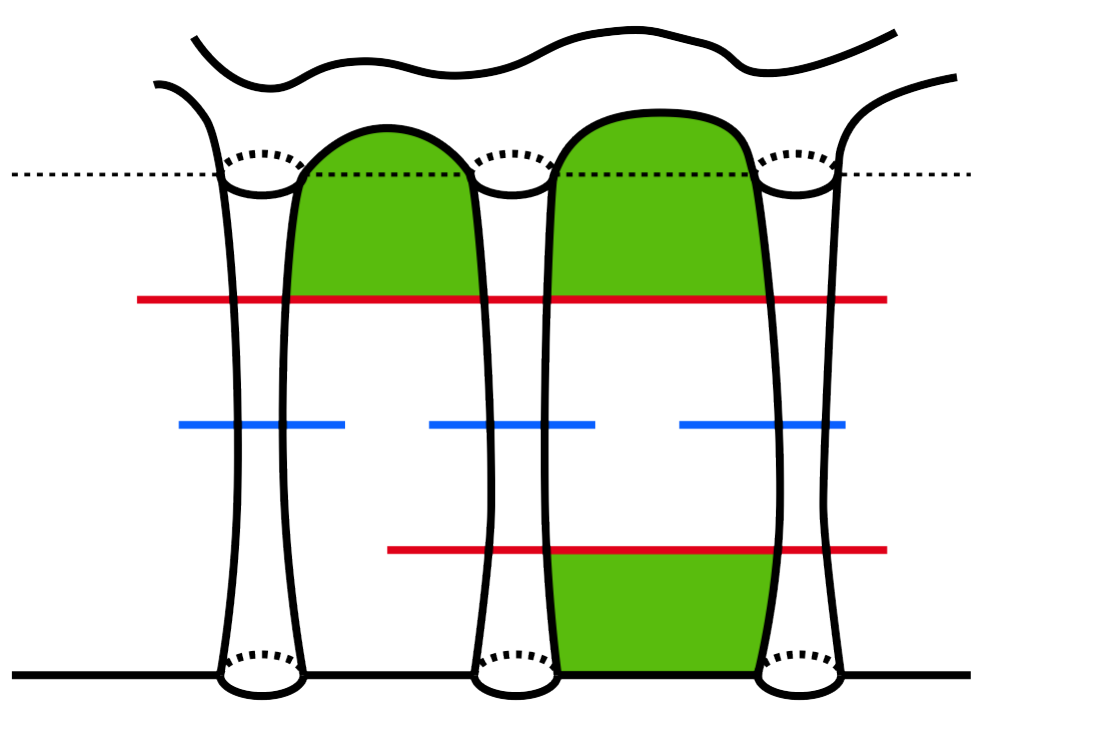}
\caption{The solid black lines show the result of amalgamating two thick surfaces (shown as horizontal dashed and solid lines) across three thin spheres (shown in blue). The horizontal red lines denote two spheres in $\mc{S}$ and the discs $X$ are shown in green. Isotoping the amalgamated thick surface across $X$ ensures it meets each red sphere in a single simple closed curve.}
\label{fig: multitube}
\end{figure}

Consequently, $\mc{H}$ is connected and we obtain $H$ (up to isotopy) by spherical self-amalgamations across a minimal set of nonseparating spheres $P$. When we perform the self-amalgamation, we first tube $\mc{H}$ to a parallel copy of $P$ and then tube $\mc{H}$ to itself via a tube that intersects each component of $P$ exactly once. Suppose that in $M \setminus P$, a component $S_0$ of $\mc{S}$ separated $\mc{H}$ from a scar of $P$. We may perform the initial tubing of $\mc{H}$ to $P$ so that the entirety of $S_0$ remains disjoint from $\mc{H}$ and then when we tube $\mc{H}$ to itself, each tube will intersect $S_0$ at most once. If some components of $\mc{S}$ intersect multiple tubes, an isotopy as in the previous paragraph ensures that we can arrange for $H$ to intersect each component of $\mc{S}$ in a single simple closed curve. \qed

\section*{Acknowledgements} I'm grateful to Alex Zupan for encouraging me to formulate and prove the Structure Theorem for Weakly Reducible Heegaard Splittings and to Marty Scharlemann for helpful comments and questions concerning early drafts of this paper.  Thanks also to the anonymous referees for helpful comments and careful reading. My perspective on (generalized) Heegaard splittings has been shaped by years of conversations with Maggy Tomova. I am thankful for research support from Colby College and National Science Foundation Grant DMS-2104022.

\section*{Statements and Declarations} I declare I have no conflict of interest.


\newpage

\begin{bibdiv}
\begin{biblist}

\bib{CG}{article}{
   author={Casson, A. J.},
   author={Gordon, C. McA.},
   title={Reducing Heegaard splittings},
   journal={Topology Appl.},
   volume={27},
   date={1987},
   number={3},
   pages={275--283},
   issn={0166-8641},
   review={\MR{0918537}},
   doi={10.1016/0166-8641(87)90092-7},
}

\bib{G3}{article}{
   author={Gabai, David},
   title={Foliations and the topology of $3$-manifolds. III},
   journal={J. Differential Geom.},
   volume={26},
   date={1987},
   number={3},
   pages={479--536},
   issn={0022-040X},
   review={\MR{0910018}},
}

\bib{Haken}{article}{
   author={Haken, Wolfgang},
   title={Some results on surfaces in $3$-manifolds},
   conference={
      title={Studies in Modern Topology},
   },
   book={
      series={Studies in Mathematics},
      volume={Vol. 5},
      publisher={Math. Assoc. America, },
   },
   date={1968},
   pages={39--98},
   review={\MR{0224071}},
}

\bib{HW}{article}{
   author={Hatcher, Allen},
   author={Wahl, Nathalie},
   title={Stabilization for mapping class groups of 3-manifolds},
   journal={Duke Math. J.},
   volume={155},
   date={2010},
   number={2},
   pages={205--269},
   issn={0012-7094},
   review={\MR{2736166}},
   doi={10.1215/00127094-2010-055},
}


\bib{HS}{article}{
   author={Hayashi, Chuichiro},
   author={Shimokawa, Koya},
   title={Thin position of a pair (3-manifold, 1-submanifold)},
   journal={Pacific J. Math.},
   volume={197},
   date={2001},
   number={2},
   pages={301--324},
   issn={0030-8730},
   review={\MR{1815259}},
   doi={10.2140/pjm.2001.197.301},
}


\bib{HeSc}{article}{
   author={Hensel, Sebastian},
      author={Schultens, Jennifer},
   title={Strong Haken via Sphere Complexes},
eprint={arXiv:2102.09831},
date = {2021}
}

\bib{Johannson}{article}{
   author={Johannson, Klaus},
   title={On surfaces and Heegaard surfaces},
   journal={Trans. Amer. Math. Soc.},
   volume={325},
   date={1991},
   number={2},
   pages={573--591},
   issn={0002-9947},
   review={\MR{1064268}},
   doi={10.2307/2001640},
}


\bib{Kobayashi}{article}{
   author={Kobayashi, Tsuyoshi},
   title={Heights of simple loops and pseudo-Anosov homeomorphisms},
   conference={
      title={Braids},
      address={Santa Cruz, CA},
      date={1986},
   },
   book={
      series={Contemp. Math.},
      volume={78},
      publisher={Amer. Math. Soc., Providence, RI},
   },
   isbn={0-8218-5088-1},
   date={1988},
   pages={327--338},
   review={\MR{0975087}},
   doi={10.1090/conm/078/975087},
}

\bib{Kobayashi2}{article}{
   author={Kobayashi, Tsuyoshi},
   title={Heegaard splittings of exteriors of two bridge knots},
   journal={Geom. Topol.},
   volume={5},
   date={2001},
   pages={609--650},
   issn={1465-3060},
   review={\MR{1857522}},
   doi={10.2140/gt.2001.5.609},
}

\bib{Lackenby}{article}{
   author={Lackenby, Marc},
   title={An algorithm to determine the Heegaard genus of simple 3-manifolds
   with nonempty boundary},
   journal={Algebr. Geom. Topol.},
   volume={8},
   date={2008},
   number={2},
   pages={911--934},
   issn={1472-2747},
   review={\MR{2443101}},
   doi={10.2140/agt.2008.8.911},
}

\bib{Moriah}{article}{
   author={Moriah, Yoav},
   title={On boundary primitive manifolds and a theorem of Casson-Gordon},
   journal={Topology Appl.},
   volume={125},
   date={2002},
   number={3},
   pages={571--579},
   issn={0166-8641},
   review={\MR{1935173}},
   doi={10.1016/S0166-8641(01)00303-0},
}

\bib{MS}{article}{
   author={Moriah, Yoav},
   author={Sedgwick, Eric},
   title={Closed essential surfaces and weakly reducible Heegaard splittings
   in manifolds with boundary},
   journal={J. Knot Theory Ramifications},
   volume={13},
   date={2004},
   number={6},
   pages={829--843},
   issn={0218-2165},
   review={\MR{2088748}},
   doi={10.1142/S0218216504003470},
}

\bib{RS}{article}{
   author={Rubinstein, Hyam},
   author={Scharlemann, Martin},
   title={Comparing Heegaard splittings of non-Haken $3$-manifolds},
   journal={Topology},
   volume={35},
   date={1996},
   number={4},
   pages={1005--1026},
   issn={0040-9383},
   review={\MR{1404921}},
   doi={10.1016/0040-9383(95)00055-0},
}

\bib{Scharlemann}{article}{
   author={Scharlemann, Martin},
   title={A Strong Haken's Theorem},
   journal={Algebraic \& Geometric Topology},
   status={to appear}
}


\bib{ST}{article}{
   author={Scharlemann, Martin},
   author={Thompson, Abigail},
   title={Thin position for $3$-manifolds},
   conference={
      title={Geometric topology},
      address={Haifa},
      date={1992},
   },
   book={
      series={Contemp. Math.},
      volume={164},
      publisher={Amer. Math. Soc., Providence, RI},
   },
   isbn={0-8218-5182-9},
   date={1994},
   pages={231--238},
   review={\MR{1282766}},
   doi={10.1090/conm/164/01596},
}


\bib{SSS}{book}{
   author={Scharlemann, Martin},
   author={Schultens, Jennifer},
   author={Saito, Toshio},
   title={Lecture notes on generalized Heegaard splittings},
   note={Three lectures on low-dimensional topology in Kyoto},
   publisher={World Scientific Publishing Co. Pte. Ltd., Hackensack, NJ},
   date={2016},
   pages={viii+130},
   isbn={978-981-3109-11-7},
   review={\MR{3585907}},
   doi={10.1142/10019},
}

\bib{ST-classification}{article}{
   author={Scharlemann, Martin},
   author={Thompson, Abigail},
   title={Heegaard splittings of $({\rm surface})\times I$ are standard},
   journal={Math. Ann.},
   volume={295},
   date={1993},
   number={3},
   pages={549--564},
   issn={0025-5831},
   review={\MR{1204837}},
   doi={10.1007/BF01444902},
}

\bib{ST-Wald}{article}{
   author={Scharlemann, Martin},
   author={Thompson, Abigail},
   title={Thin position and Heegaard splittings of the $3$-sphere},
   journal={J. Differential Geom.},
   volume={39},
   date={1994},
   number={2},
   pages={343--357},
   issn={0022-040X},
   review={\MR{1267894}},
}


\bib{Schultens}{article}{
   author={Schultens, Jennifer},
   title={The classification of Heegaard splittings for (compact orientable
   surface)$\,\times\, S^1$},
   journal={Proc. London Math. Soc. (3)},
   volume={67},
   date={1993},
   number={2},
   pages={425--448},
   issn={0024-6115},
   review={\MR{1226608}},
   doi={10.1112/plms/s3-67.2.425},
}



\bib{Taylor}{article}{
   author={Taylor, Scott A.},
   title={Equivariant Heegaard genus of reducible 3-manifolds},
   journal={Math. Proc. Cambridge Philos. Soc.},
   volume={175},
   date={2023},
   number={1},
   pages={51--87},
   issn={0305-0041},
   review={\MR{4600195}},
   doi={10.1017/S0305004123000038},
}

\bib{TT1}{article}{
   author={Taylor, Scott A.},
   author={Tomova, Maggy},
   title={Thin position for knots, links, and graphs in 3-manifolds},
   journal={Algebr. Geom. Topol.},
   volume={18},
   date={2018},
   number={3},
   pages={1361--1409},
   issn={1472-2747},
   review={\MR{3784008}},
   doi={10.2140/agt.2018.18.1361},
}

\bib{Waldhausen-3sph}{article}{
   author={Waldhausen, Friedhelm},
   title={Heegaard-Zerlegungen der $3$-Sph\"{a}re},
   language={German},
   journal={Topology},
   volume={7},
   date={1968},
   pages={195--203},
   issn={0040-9383},
   review={\MR{0227992}},
   doi={10.1016/0040-9383(68)90027-X},
}

\bib{Waldhausen-sufflarge}{article}{
   author={Waldhausen, Friedhelm},
   title={On irreducible $3$-manifolds which are sufficiently large},
   journal={Ann. of Math. (2)},
   volume={87},
   date={1968},
   pages={56--88},
   issn={0003-486X},
   review={\MR{0224099}},
   doi={10.2307/1970594},
}

\end{biblist}
\end{bibdiv}
\end{document}